\documentclass[a4paper,11pt,reqno]{amsart}

\usepackage[T1]{fontenc}
\usepackage[utf8]{inputenc}
\usepackage{lmodern}
\usepackage[english]{babel}

\usepackage[%
	left=2.5cm,       
	right=2.5cm,      
	top=3.5cm,        
	bottom=3.5cm,     
	heightrounded,    
	bindingoffset=0mm 
]{geometry}

\usepackage{amsmath}
\usepackage{amssymb}
\usepackage{mathtools}
\usepackage{mathrsfs}
\usepackage{xfrac}
\numberwithin{equation}{section}

\usepackage[hidelinks]{hyperref} 
\usepackage[noabbrev,capitalize]{cleveref}

\usepackage{bookmark}

\theoremstyle{plain}
	\newtheorem{theorem}{Theorem}[section]
	
	\newtheorem*{lemma*}{Lemma}
	
	\newtheorem{corollary}[theorem]{Corollary}

\theoremstyle{definition}
	\newtheorem{definition}[theorem]{Definition}
	
	\newtheorem{remark}[theorem]{Remark}
	\newtheorem{assumption}[theorem]{Assumption}

\usepackage[initials]{amsrefs}

\usepackage[dvipsnames]{xcolor}
  
\usepackage{enumerate}

\usepackage{url}

\newcommand{\de}{\partial}\renewcommand{\div}{\mathrm{div}\,} 
\newcommand{\N}{\mathbb{N}}
\newcommand{\R}{\mathbb{R}}
\newcommand{\T}{\mathbb{T}}
\newcommand{\e}{\mathrm{e}}
\newcommand{\eps}{\varepsilon}

\renewcommand{\d}{\mathsf{d}}
\renewcommand{\phi}{\varphi}
\renewcommand{\rho}{\varrho}
\renewcommand{\theta}{\vartheta}

\DeclareMathOperator{\loc}{loc}
\DeclareMathOperator{\ul}{ul}
\DeclareMathOperator{\di}{\,d\!}
\DeclareMathOperator{\curl}{curl}

\DeclareMathOperator{\supp}{supp}

\DeclarePairedDelimiter{\set}{\{}{\}}

\mathchardef\ordinarycolon\mathcode`\:
\mathcode`\:=\string"8000
\begingroup \catcode`\:=\active
  \gdef:{\mathrel{\mathop\ordinarycolon}}
\endgroup

\def\Xint#1{\mathchoice
{\XXint\displaystyle\textstyle{#1}}%
{\XXint\textstyle\scriptstyle{#1}}%
{\XXint\scriptstyle\scriptscriptstyle{#1}}%
{\XXint\scriptscriptstyle\scriptscriptstyle{#1}}%
\!\int}
\def\XXint#1#2#3{{\setbox0=\hbox{$#1{#2#3}{\int}$ }
\vcenter{\hbox{$#2#3$ }}\kern-.6\wd0}}

\def\aint{\Xint-}

\allowdisplaybreaks

\begin{document}

\title[Euler flow in localized Yudovich spaces]{An elementary proof of existence and uniqueness \\ for the Euler flow in localized Yudovich spaces}

\author[G.~Crippa]{Gianluca Crippa}
\address[G.~Crippa]{Department Mathematik und Informatik, Universit\"at Basel, Spiegelgasse 1, CH-4051 Basel, Switzerland}
\email{gianluca.crippa@unibas.ch}

\author[G.~Stefani]{Giorgio Stefani}
\address[G.~Stefani]{Department Mathematik und Informatik, Universit\"at Basel, Spiegelgasse 1, CH-4051 Basel, Switzerland}
\email{giorgio.stefani@unibas.ch}

\date{\today}



\thanks{
\textit{Acknowledgements}. 
This research has been partially supported by the ERC Starting Grant 676675 FLIRT -- \textit{Fluid Flows and Irregular Transport} and by the SNF Project 212573 FLUTURA -- \textit{Fluids, Turbulence, Advection.}
The second author is member of the Istituto Nazionale di Alta Matematica (INdAM), Gruppo Nazionale per l'Analisi, la Probabilità e le loro Applicazioni (GNAMPA), has been partially supported by the INdAM--GNAMPA Project 2020 \textit{Problemi isoperimetrici con anisotropie} (n.\ prot.\ U-UFMBAZ-2020-000798 15-04-2020), is partially supported by the INdAM--GNAMPA 2022 Project \textit{Analisi geometrica in strutture subriemanniane}, codice CUP\_E55\-F22\-000\-270\-001 and by the INdAM--GNAMPA 2023 Project \textit{Problemi variazionali per funzionali e operatori non-locali}, codice CUP\_E53\-C22\-001\-930\-001, and has received funding from the European Research Council (ERC) under the European Union’s Horizon 2020 research and innovation program (grant agreement No.~945655). 
The authors thank M.~Inversi for his careful reading of the manuscript and for several precious comments that helped to improve the exposition.
}

\begin{abstract}
We revisit Yudovich's well-posedness result for the $2$-dimensional Euler equations for an inviscid incompressible fluid on either a sufficiently regular (not necessarily bounded) open set $\Omega\subset\R^2$ or on the torus $\Omega=\T^2$.

We construct global-in-time weak solutions with vorticity in $L^1\cap L^p_{\ul}$ and in $L^1\cap Y^\Theta_{\ul}$, where $L^p_{\ul}$ and $Y^\Theta_{\ul}$ are suitable uniformly-localized versions of the Lebesgue space $L^p$ and of the Yudovich space $Y^\Theta$ respectively, with no condition at infinity for the growth function~$\Theta$.
We also provide an explicit modulus of continuity for the velocity depending on the growth function~$\Theta$. We prove uniqueness of weak solutions in $L^1\cap Y^\Theta_{\ul}$ under the assumption that~$\Theta$ grows moderately at infinity. In contrast to Yudovich's energy method, we employ a Lagrangian strategy to show uniqueness.

Our entire argument relies on elementary real-variable techniques, with no use of either Sobolev spaces, Calder\'on--Zygmund theory or Littlewood--Paley decomposition, and actually applies not only to the Biot--Savart law, but also to more general operators whose kernels obey some natural structural assumptions.
\end{abstract}

\maketitle

\section{Introduction}

\subsection{Euler equations}
The two-dimensional Euler equations for an incompressible inviscid fluid are given by 
\begin{equation}
\label{eq:vEuler}
\begin{cases}
\de_t v+(v\cdot\nabla)v+\nabla p=0
&
\text{in}\ (0,+\infty)\times\Omega,\\
\div v=0	
&
\text{in}\ [0,+\infty)\times\Omega,\\
v\cdot\nu_\Omega=0	
&
\text{on}\ [0,+\infty)\times\de\Omega,\\
v|_{t=0}=v_0
&
\text{on}\ \Omega,\\
\end{cases}
\end{equation}
on either a sufficiently smooth (possibly unbounded) simply connected open set $\Omega\subset\R^2$ or on the $2$-dimensional torus $\Omega=\T^2$, 
where $v\colon[0,+\infty)\times\Omega\to\R^2$ is the \emph{velocity} of the fluid, $p\colon[0,+\infty)\times\Omega\to\R$ is the (scalar) \emph{pressure} and $\nu_\Omega\colon\de\Omega\to\R^2$ is the inner unit normal to $\de\Omega$.
In the cases $\Omega=\R^2$ and $\Omega=\T^2$, no boundary condition is imposed.

The \emph{vorticity} $\omega\colon[0,+\infty)\times\Omega\to\R$ of the fluid is given by the relation $\omega=\curl v$ and satisfies the Euler equations in \emph{vorticity form}
\begin{equation}
\label{eq:Euler}
\begin{cases}
\de_t\omega + \div(v\omega)=0 
& \text{in}\ (0,+\infty)\times\Omega,\\[1mm]
v=K\omega 
& \text{in}\ [0,+\infty)\times\Omega,\\[1mm]
\omega|_{t=0}=\omega_0 & \text{on}\ \Omega.
\end{cases}	
\end{equation}
The relation appearing in the second line of~\eqref{eq:Euler} is the so-called \emph{Biot--Savart law} and allows to recover the velocity~$v$ from the vorticity~$\omega$.
In fact, since $\div v=0$, there exists a \emph{stream function} $\psi\colon[0,+\infty)\times\Omega\to\R$ (uniquely determined up to an additive time-dependent constant, if $\Omega$ is connected) such that
\begin{equation}
\label{eq:stream}
v
=
\begin{pmatrix}
-\de_{x_2}\psi\\
\de_{x_1}\psi
\end{pmatrix}
=
\nabla^\perp\psi
\quad
\text{on}\
[0,+\infty)\times\Omega.
\end{equation}
By applying the $\curl$ operator to both sides of~\eqref{eq:stream}, we get the Poisson equation
\begin{equation}
\label{eq:poisson}
\Delta\psi = \omega
\quad
\text{on}\ \Omega,
\end{equation}
so that 
\begin{equation*}
v(t,x)
=
\int_\Omega k(x,y)\,\omega(t,y)\di y
=
K\omega(t,x)
\end{equation*} 
for $x\in\Omega$ and $t\in[0,+\infty)$, where $k\colon\Omega\times\Omega\to\R^2$ is an integral kernel obtained by composing the operator $\nabla^\perp$ with the Newtonian potential on~$\Omega$. Note that the relation $v=K\omega$ encodes both the incompressibility property of the fluid $\div v=0$ and the \emph{no-flow boundary condition} $v\cdot\nu_\Omega=0$, since one imposes a Dirichlet condition at the boundary of~$\Omega$ in order to solve the Poisson equation~\eqref{eq:poisson}. 
Also note that, in the case of the $2$-dimensional torus, \eqref{eq:poisson} is only solvable under the compatibility condition that $\omega$ has zero average on~$\T^2$.
At least formally, such condition follows from the definition of $\omega$ as the $\curl$ of the velocity field. In case the set~$\Omega$ is not simply connected, the Biot--Savart law needs to be modified taking into account the circulations of the velocity field around the ``holes'' of $\Omega$, which requires the use of suitable harmonic vector fields, see~\cite{ILL20} for instance.

If $\Omega=\R^2$, then actually $k(x,y)=k_2(x-y)$ with
\begin{equation*}
k_2(x)
=
\frac1{2\pi}\,\frac{1}{|x|^2}
\begin{pmatrix}
-x_2\\
x_1	
\end{pmatrix}
=
\frac1{2\pi}\,\frac{x^\perp}{|x|^2}
\quad
\text{for all}\ x\in\R^2,\ x\ne0.
\end{equation*}
On a sufficiently regular open set $\Omega\subset\R^2$ and on the torus $\Omega=\T^2$, the kernel $k$ does not have such an easy and explicit expression but, nevertheless, is known to satisfy some suitable a priori estimates (see~\cites{MP84,MP94} and also inequalities~\eqref{eq:k_kernel_bound} and~\eqref{eq:k_kernel_oscillation_bound} below).

For a detailed exposition of the theory of the Euler equations,
we refer the reader to the monographs~\cites{BCD11,C95,MB02,MP94} and to the survey~\cite{C07}. 

Existence and uniqueness of \emph{weak solutions} of~\eqref{eq:Euler} with bounded vorticity is due to Yudovich~\cite{Y63}.
Existence of weak solutions was later achieved even for unbounded vorticities under weaker integrability assumptions, see~\cites{D91,DM87,EM94,M93,VW93} for the most relevant results.

Uniqueness of unbounded weak solutions of~\eqref{eq:Euler} is an extremely delicate problem. On the one side, in~\cite{Y95} Yudovich himself extended his previous uniqueness result~\cite{Y63} to the case of unbounded vorticities belonging to the now-called \emph{Yudovich space}, see below for the precise definition.
A different approach relying on Littlewood--Paley decomposition techniques was pursued by Vishik~\cite{V99}.
Further improvements were subsequently obtained by several authors~\cites{BH15,BK14,CMZ19,HK08}, additionally establishing some propagation of regularity of solutions under more restrictive assumptions on the initial data.   
Important results have been also achieved on open sets with rough boundary, see~\cites{HZ20,L15,LMW14,LZ19}.
On the other side, the uniqueness of weak solutions of~\eqref{eq:Euler} in~$L^p(\R^2)$ for $p<+\infty$ is currently an open problem, see~\cites{BM20,BS21,V18-I,V18-II,BC21} for some recent advances.

\subsection{Yudovich's energy method}

In this paper, we revisit Yudovich's well-posedness result in~\cite{Y95}.
Our approach is simpler and explicit, and is based on elementary real-variable techniques only.
In fact, we make no use of either Fourier theory or Littlewood--Paley decomposition and even, somewhat surprisingly, we do not need to rely on either Sobolev spaces or Calder\'on--Zygmund operator theory.

Yudovich's original approach~\cites{Y63,Y95} to the uniqueness is essentially based on a clever {\em energy argument} (we refer the reader also to~\cite{C95}*{Chapter~5}, \cite{MB02}*{Chapter~8} and~\cite{MP94}*{Chapter~2} for a more detailed exposition).
The idea behind this method is to show that the squared $L^2$ distance between the velocities of two solutions (also called the {\em relative energy}) 
\begin{equation*}
E(t)
=
\int_\Omega|v(t,x)-\tilde v(t,x)|^2\di x
\end{equation*}
starting from the same initial datum obeys a Gr\"onwall-type integral inequality.

If the vorticity is bounded, then one can exploit the Biot--Savart law $v=K\omega$ in~\eqref{eq:Euler} together with some standard Calder\'on--Zygmund estimates to get
\begin{equation}
\label{eq:yudo_v_bounded}
\|\nabla v\|_{L^p(\Omega)}
\le
Cp
\end{equation}
for any $p\in(1,+\infty)$ sufficiently large, where the constant $C>0$ depends on $\|\omega\|_{L^\infty(\Omega)}$ only. 
An energy estimate on~\eqref{eq:vEuler} combined with~\eqref{eq:yudo_v_bounded} gives
\begin{equation}
\label{eq:yudo_E_bounded}
\frac{\di}{\di t}\,E(t)
\le
Cp\,E(t)^{1-\sfrac1p}
\quad
\text{for}\ t\in[0,T].
\end{equation}
By comparison with the maximal solution of~\eqref{eq:yudo_E_bounded}, one must have that 
\begin{equation*}
E(t)
\le 
(Ct)^p
\le 
(CT)^p
\quad
\text{for}\ t\in[0,T],
\end{equation*} 
so that $E(t)=0$ for all $t\in[0,T]$ letting $p\to+\infty$, provided that $CT<1$.

If the vorticity is not bounded but the function $p\mapsto\|\omega\|_{L^p(\Omega)}$ has moderate growth for $p\to+\infty$, then the argument above can be modified to get 
an estimate of the form
\begin{equation}
\label{eq:yudo_E_Y}
E(t)
\lesssim
\int_0^t\beta(E(s))\di s,
\end{equation}
for some function $\beta\colon[0,+\infty)\to[0,+\infty)$ depending on the growth of the $L^p$ norm of the vorticity for $p\to+\infty$, namely, to which \emph{Yudovich space} the vorticity belongs to.

Let us recall the  definition of \emph{Yudovich space}.
Here and in the rest of the paper, we let $\Theta\colon [1,+\infty)\to(0,+\infty)$ be a non-decreasing function. 

\begin{definition}[Yudovich space] 
We let
\begin{equation*}
Y^\Theta(\Omega)
=
\set*{f\in\bigcap_{p\in[1,+\infty)} L^p(\Omega) \; : \;  
\|f\|_{Y^\Theta(\Omega)}
=
\sup_{p\in[1,+\infty)}\frac{\|f\|_{L^p(\Omega)}}{\Theta(p)}<+\infty}
\end{equation*}
be the \emph{Yudovich space} on~$\Omega$ associated to~$\Theta$.	
\end{definition}

Note that, if $\Theta$ is bounded, then $Y^\Theta(\Omega)\subset L^\infty(\Omega)$.
If $\Theta$ is unbounded, then it is not difficult to see that $Y^\Theta(\Omega)\not\subset L^\infty(\Omega)$.

Now, if the vorticity belongs to $Y^\Theta(\Omega)$, then one replaces~\eqref{eq:yudo_v_bounded} with
\begin{equation*}
\|\nabla v\|_{L^p(\Omega)}
\lesssim
\|\omega\|_{Y^\Theta(\Omega)}\,p\,\Theta(p)
\end{equation*}
and the computation leading to~\eqref{eq:yudo_E_bounded} now gives
\begin{equation*}
\frac{\di}{\di t}E(t) \lesssim \frac{1}{\eps} \|\omega\|_{L^{\sfrac{1}{\eps}}(\Omega)} E(t)^{1-\eps} \lesssim E(t) \frac{1}{\eps} \Theta(\sfrac{1}{\eps}) \frac{1}{E(t)^\eps}
\end{equation*}
for $\eps>0$, where the implicit constant depends on the $Y^\Theta$ norm of the vorticity. Setting
\begin{equation}
\label{eq:yudo_inf_function}
\psi_\Theta(r)
=
\begin{cases}
\inf\set*{\frac{1}{\eps}\, \Theta\left(\sfrac1\eps\right) \;:\; \eps\in\left(0,\sfrac1{3}\right)}
&
\text{for}\
r\in[0,1),
\\[3mm]
\inf\set*{\frac{1}{\eps}\,\Theta\left(\sfrac1\eps\right)\,r^\eps \;:\; \eps\in\left(0,\sfrac1{3}\right)}
&
\text{for}\
r\in[1,+\infty)
\end{cases}
\end{equation} 
(here the choice of the value $\sfrac13$ is irrelevant and is made for convenience only), we finally obtain
\begin{equation*}
\frac{\di}{\di t}E(t) \lesssim E(t) \, \psi_\Theta\left(\frac{1}{E(t)} \right),
\end{equation*}
where the implicit constant depends on the $Y^\Theta$ norm of the vorticity. We have therefore obtained~\eqref{eq:yudo_E_Y} with $\beta(z)=z \, \psi_\Theta(\sfrac{1}{z})$. Based on this computation, Yudovich's well-posedness result can be stated as follows (for the precise notion of \emph{weak solution} of~\eqref{eq:Euler}, see \cref{def:weak_solution} below).

\begin{theorem}[Yudovich~\cites{Y63,Y95}]
\label{res:Yudovich}
Let $\Omega\subset\R^2$ be a bounded open set with $C^2$ boundary and assume that the function $\psi_\Theta$ in~\eqref{eq:yudo_inf_function} satisfies
\begin{equation}
\label{eq:yudo_Osgood}
\int_{0^+}
\frac{\di r}{r\,\psi_\Theta(\sfrac{1}{r})}
=
+\infty.
\end{equation}
Then, for any initial datum $\omega_0\in  Y^\Theta(\Omega)$, there exists a unique weak solution $(\omega,v)$ of~\eqref{eq:Euler} such that
\begin{equation}
\label{eq:res_Yudovich}
\omega\in 
L^\infty([0,+\infty); Y^\Theta(\Omega)),
\quad
v\in
L^\infty([0,+\infty);C_b(\Omega;\R^2)).
\end{equation}
\end{theorem}

In~\cite{Y95}*{Theorem~2}, Yudovich also proved that the velocity $v$ in~\eqref{eq:res_Yudovich} satisfies 
\begin{equation}
\label{eq:res_Yudovich_psi_Theta}
|v(t,x)-v(t,y)|
\lesssim
|x-y|\,\psi_\Theta(\sfrac{1}{|x-y|^3})
\quad
\forall x,y\in\Omega 
\end{equation} 
for a.e.~$t>0$, and observed that the modulus of continuity $r\mapsto r\,\psi_\Theta(\sfrac{1}{r^3})$ satisfies the \emph{Osgood condition}
\begin{equation*}
\int_{0^+}\frac{\di r}{r\,\psi_\Theta(\sfrac{1}{r^3})}=+\infty,
\end{equation*}
as a consequence of~\eqref{eq:yudo_Osgood}

As it is apparent from the definition in~\eqref{eq:yudo_inf_function}, the precise behavior of~$\psi_\Theta$ and its dependence on the growth function~$\Theta$ are quite implicit.
As a matter of fact, in his paper~\cite{Y95} Yudovich exhibited explicit formulas for the function~$\psi_\Theta$ only in some particular cases.
Precisely, if 
\begin{equation}
\label{eq:yudo_Theta_example}
\Theta_m(p)
=
\log p\,
\log_2 p\,
\log_3 p\,
\cdots\,
\log_m p
\end{equation}
for some $m\in\N$ and for all $p\in(1,+\infty)$ sufficiently large, where
\begin{equation*}
\log_m p
=
\underbrace{\log\log\dots\log}_{\text{$m$ times}} p ,
\end{equation*}
then
\begin{equation*}
\psi_{\Theta_m}(r)
\lesssim
\log r\,
\log_2 r\,
\log_3 r\,
\cdots\,
\log_{m+1} r
\end{equation*}
for $r>0$ sufficiently large.
In this range of examples, condition~\eqref{eq:yudo_Osgood} holds for all $m\in\N$.
Condition~\eqref{eq:yudo_Osgood} however fails for a growth function of order $\Theta(p)\eqsim p$ for $p\to+\infty$.
In other words, as observed in~\cite{Y95}*{Examples~3.2 and~3.3}, \cref{res:Yudovich} holds for vorticities with singularities of order $|\log|\log|x||$ (corresponding to a growth function of order $\Theta(p)\eqsim\log p$), but not for vorticities with singularities of order $|\log|x||$ (corresponding to a growth function of order $\Theta(p)\eqsim p$) which, in turn, are typical singularities of $\mathrm{BMO}$ functions, see the discussion in~\cite{V99} and the estimate~\eqref{eq:bmo} below. 

\subsection{Uniformly-localized \texorpdfstring{$L^p$}{Lˆp} and Yudovich spaces}

As recently pointed out by the work of Taniuchi~\cite{T04} and by the subsequent developments obtained in~\cites{TTY10,CMZ19}, Yudovich's approach can be suitably localized in order to treat vorticities with possibly infinite global $L^1$ norm.

Let us recall the definition of the \emph{uniformly-localized} version of the Yudovich space introduced above.
Here and in the rest of the paper, we let $\d\colon\Omega\times\Omega\to[0,+\infty)$ be the natural distance on~$\Omega$, that is, the Euclidean distance if $\Omega\subset\mathbb\R^2$ and the geodesic distance if $\Omega=\T^2$. 
We let $B_r(x)$ be the open ball of radius $r>0$ centered at $x\in\R^2$.

\begin{definition}[Uniformly-localized $L^p$ and Yudovich spaces]
Let $p\in[1,+\infty)$.
We let 
\begin{equation}
\label{eq:def_Lp_ul}
L^p_{\ul}(\Omega)
=
\set*{
f\in L^p_{\loc}(\Omega) \;:\; 
\|f\|_{L^p_{\ul}(\Omega)}
=
\sup_{x\in\Omega}\|f\|_{L^p(\Omega\cap B_1(x))}
<+\infty}
\end{equation}
be the \emph{uniformly-localized $L^p$ space} on~$\Omega$.
By convention, we set $L^\infty_{\ul}(\Omega)=L^\infty(\Omega)$.
We also let
\begin{equation*}
Y^\Theta_{\ul}(\Omega)
=
\set*{f\in\bigcap_{p\in[1,+\infty)} L^p_{\ul}(\Omega) \;:\; 
\|f\|_{Y^\Theta_{\ul}(\Omega)}
=
\sup_{p\in[1,+\infty)}
\frac
{\|f\|_{L^p_{\ul}(\Omega)}}
{\Theta(p)}<+\infty}
\end{equation*}
be the \emph{uniformly-localized Yudovich space} on~$\Omega$ associated to~$\Theta$.	
\end{definition}

Clearly, we have $Y^\Theta(\Omega)\subset Y^\Theta_{\ul}(\Omega)$, with strict inclusion if $\Omega$ is unbounded.
Note that, by an elementary geometric argument, if we set
\begin{equation*}
\|f\|_{L^p_{\ul,r}(\Omega)}
=
\sup_{x\in\Omega}\|f\|_{L^p(\Omega\cap B_r(x))}
\end{equation*}
for all $r>0$, then
\begin{equation}
\label{eq:Lp_ul_change_radius}
\|f\|_{L^p_{\ul,R}(\Omega)}
\lesssim
\left(\frac Rr\right)^{2/p}
\|f\|_{L^p_{\ul,r}(\Omega)}
\end{equation}
for all $p\in[1,+\infty)$ and $R>r>0$.
In particular, the choice $r=1$ made in the definition~\eqref{eq:def_Lp_ul} of the space $L^p_{\ul}(\Omega)$ is completely irrelevant.  

With these definitions at hand, Taniuchi's well-posedness result can be stated as follows (see~\cite{TTY10} for a similar result dealing with almost-periodic initial data in~$\R^2$ and~\cite{CMZ19}*{Theorem~1.10} for initial data additionally belonging to a suitable \emph{Spanne space}). 

\begin{theorem}[Taniuchi~\cites{T04,CMZ19}]
\label{res:Taniuchi}
Let $\Theta\colon [1,+\infty)\to(0,+\infty)$ be a non-decreasing function such that
\begin{equation}
\label{eq:res_Taniuchi_Theta}
\int^{+\infty}\frac{\di p}{p\,\Theta(\log p)}=+\infty.
\end{equation}
Then, for any initial datum $\omega_0\in Y^\Theta_{\ul}(\R^2)$, there exists a weak solution $(\omega,v)$ of~\eqref{eq:Euler} such that
\begin{equation}
\label{eq:res_Taniuchi}
\omega\in 
L^\infty_{\loc}([0,+\infty),Y^\Theta_{\ul}(\R^2)),
\quad
v\in
L^\infty_{\loc}([0,+\infty);L^\infty_{\loc}(\R^2;\R^2)).
\end{equation}
In addition, if $\Theta$ satisfies
\begin{equation}
\label{eq:Yudovich_condition_Theta}
\int^{+\infty}\frac{\di p}{p\,\Theta(p)}=+\infty,
\end{equation}
then the solution $(\omega,v)$ in~\eqref{eq:res_Taniuchi} is unique.
\end{theorem}

Note that condition~\eqref{eq:res_Taniuchi_Theta} is satisfied by a growth function $\Theta(p)\eqsim p$ for \mbox{$p\to+\infty$}. 
In particular, since
\begin{equation}
\label{eq:bmo}	
\|f\|_{L^p_{\ul}(\R^2)}
\lesssim
p \, \|f\|_{\mathrm{bmo}(\R^2)}
\end{equation}
for all $p\in(1,+\infty)$ (see~\cite{T04}*{Definitions~3 and 5}), \cref{res:Taniuchi} provides existence of weak solutions of~\eqref{eq:Euler} starting from a $\mathrm{BMO}$ initial vorticity and bounded initial velocity (although in general the solution does not belong to $\mathrm{BMO}$ at later times), improving the previous existence result by Vishik~\cite{V99}.
We refer the reader to~\cite{T04}*{Corollary~1.2} for the precise (and more general) statement of this result.

\subsection{Main results}
In this paper, we first of all completely revisit Yudovich's uniqueness result (\cref{res:Yudovich}), employing an elementary and direct approach
which makes the relation among the growth of the $L^p$ norm of the vorticity, the modulus of continuity of the associated velocity, and the condition required for the uniqueness fully explicit.

Before stating our uniqueness result, let us introduce some notation that will be used throughout the paper.

\begin{definition}[The function $\phi_\Theta$]
Let $\Theta\colon [1,+\infty)\to(0,+\infty)$ be a non-decreasing function. 
We let the function $\phi_\Theta\colon[0,+\infty)\to[0,+\infty)$ be such that $\phi_\Theta(0)=0$ and
\begin{equation}
\label{eq:def_phi_Theta}
\phi_\Theta(r)
=
\begin{cases}
r\,(1-\log r)\,\Theta(1-\log r)
&
\text{for}\ r\in(0,\e^{-2}]\\[2mm]
\e^{-2}\,3\,\Theta(3) 
& 
\text{for}\ r>\e^{-2}
\end{cases}
\end{equation} 
(the choice of the constant $\e^{-2}$ is irrelevant and is made for convenience only, see below).
With a slight abuse of terminology, we say that $\phi_\Theta$ is the \emph{modulus of continuity associated to the growth function~$\Theta$}, and we define  
\begin{equation*}
C_b^{0,\phi_\Theta}(\Omega;\mathbb R^2)
=
\set*{v\in L^\infty(\Omega;\R^2)
\;:\;
\sup_{x \not= y}
\frac{|v(x)-v(y)|}{\phi_\Theta(\d(x,y))}
<+\infty
}.	
\end{equation*}
\end{definition}

With the above notation in force, our uniqueness result can be stated as follows (for the precise notions of \emph{weak solution} and of \emph{Lagrangian weak solution} of the system~\eqref{eq:Euler}, we again refer the reader to \cref{def:weak_solution} below).

\begin{theorem}[Uniqueness]
\label{res:main_uniqueness}
Let $\Omega\subset\R^2$ be either a sufficiently regular open set or the $2$-dimensional torus $\Omega=\T^2$.
If $\Theta$ satisfies~\eqref{eq:Yudovich_condition_Theta} and the function~$\phi_\Theta$ defined in~\eqref{eq:def_phi_Theta} 
is concave on $[0,+\infty)$, then there is at most one Lagrangian weak solution $(\omega,v)$ of~\eqref{eq:Euler} with
\begin{equation}
\label{eq:res_main_solution_Theta}
\omega\in 
L^\infty_{\loc}([0,+\infty);L^1(\Omega)\cap Y^\Theta_{\ul}(\Omega)),
\quad
v\in
L^\infty_{\loc}([0,+\infty);C_b^{0,\phi_\Theta}(\Omega;\R^2)),
\end{equation}
starting from the initial datum $\omega_0\in L^1(\Omega)\cap Y^\Theta_{\ul}(\Omega)$.	
\end{theorem}

In \cref{res:main_uniqueness}, we do not specify the required regularity of the open set $\Omega\subset\R^2$ in detail, since such regularity is only needed to ensure the well-posedness of the Biot--Savart law appearing in~\eqref{eq:Euler}.
As a matter of fact, we do not require the open set $\Omega\subset\R^2$ either to be bounded or to have finite measure, in contrast to the result in~\cref{res:Yudovich}.

Actually, \cref{res:main_uniqueness} does hold for any operator~$K$ satisfying some suitable conditions (which hold in particular in the case of the Biot--Savart law), see \cref{ass:est} and \cref{ass:diverg_non-tang} below.
 
Last but not least, the function $\phi_\Theta$ defined in~\eqref{eq:def_phi_Theta} provides a fully explicit modulus of continuity of the velocity in terms of the integrability of the vorticity.
In other words, the regularity of the velocity stated in~\eqref{eq:res_main_solution_Theta} can be seen as a more explicit version of~\eqref{eq:res_Yudovich_psi_Theta} (even for a growth function $\Theta$ possibly not  implying the uniqueness of the solution, see \cref{res:main_existence} below).
In particular, \cref{res:main_uniqueness} applies to the explicit growth function $\Theta_m$ in~\eqref{eq:yudo_Theta_example} for all $m\in\N$, for which one can easily see that
\begin{equation*}
\phi_{\Theta_m}(r)
\eqsim
r
\,
\log(\sfrac1r)
\,
\log_2(\sfrac1r)
\dots
\log_{m+1}(\sfrac1r)
\end{equation*} 
for all $r>0$ sufficiently small. 

\begin{remark}
Actually, the word ``Lagrangian'' can be removed from the statement of~\cref{res:main_uniqueness}, in the sense that uniqueness can be shown in the (a priori larger) class of all weak solutions. This is due to the fact that, for a continuity equation with an Osgood velocity field, all weak solutions in $L^1(\Omega)$ are in fact Lagrangian. This fact is not at all elementary and has been proved (via very different approaches) in~\cites{AB08,CJMO19}, also see~\cite{CC21} in the context of Sobolev velocity fields. We nevertheless prefer to state~\cref{res:main_uniqueness} for Lagrangian solutions in order to emphasize the best result that it is possible to prove just relying on our elementary approach.
\end{remark}

Concerning the existence of weak solutions of~\eqref{eq:Euler}, somewhat inspired by Taniuchi's \cref{res:Taniuchi}, we prove the following result.

\begin{theorem}[Existence]
\label{res:main_existence}
Let $\Omega\subset\R^2$ be either a sufficiently regular open set or the $2$-dimensional torus $\Omega=\T^2$ and let $p\in(2,+\infty)$.
For any initial datum $\omega_0\in L^1(\Omega)\cap L^p_{\ul}(\Omega)$, there exists a weak solution $(\omega,v)$ of~\eqref{eq:Euler} such that
\begin{equation}
\label{eq:res_main_solution_Lp_ul}
\omega\in 
L^\infty_{\loc}([0,+\infty);L^1(\Omega)\cap L^p_{\ul}(\Omega)),
\quad
v\in
L^\infty_{\loc}([0,+\infty);C_b^{0,1-\sfrac2p}(\Omega;\R^2)).
\end{equation}
Moreover, if $\omega_0\in L^1(\Omega)\cap Y^\Theta_{\ul}(\Omega)$, then the weak solution $(\omega,v)$ of~\eqref{eq:Euler} given in~\eqref{eq:res_main_solution_Lp_ul} additionally satisfies~\eqref{eq:res_main_solution_Theta} and, provided that $\Theta$ satisfies~\eqref{eq:Yudovich_condition_Theta}, is Lagrangian.
\end{theorem}

As for \cref{res:main_uniqueness} above, the regularity of the open set $\Omega\subset\R^2$ is only needed to guarantee the well-posedness of the Biot--Savart law. 
In fact, as before, also \cref{res:main_existence} applies to any operator $K$ satisfying the a priori estimates stated in \cref{ass:est} and \cref{ass:diverg_non-tang} below.

Up to our knowledge, the global-in-time existence result stated in \cref{res:main_existence} is new, even for~$K$ being the standard Biot--Savart operator.
Local-in-time existence of weak solutions of~\eqref{eq:Euler} with vorticity only in $L^p_{\ul}$ for some~$p>2$ is known for $\Omega=\R^2$, see~\cite{T04}*{Theorem~1.3}.

The global-in-time existence result for the $L^1\cap Y^\Theta_{\ul}$-spatial regularity of the vorticity in \cref{res:main_existence} does not require any assumption on the behavior at infinity of the growth function~$\Theta$.
In this sense, the global $L^1$ integrability of the vorticity allows us to remove the condition~\eqref{eq:res_Taniuchi_Theta} needed in \cref{res:Taniuchi}.

Finally, \cref{res:main_existence} provides the existence of a solution and a modulus of continuity for the velocity even for growth functions~$\Theta$ allowing for vorticities not included in the $\mathrm{BMO}$-like spaces considered by Vishik in~\cite{V99}, by Bernicot, Hmidi and Keraani in~\cites{BH15,BK14} and by Chen, Miao and Zhen in~\cite{CMZ19}.
Indeed, we can treat growth functions like $\Theta(p)\eqsim p^\alpha$ for $p\to+\infty$ for all $\alpha>0$, corresponding to singularities of order $|\log|x||^\alpha$.
In addition, since the classes considered in \cref{res:main_existence} are of integral type, our existence result allows for the cut-off of the initial datum, a property which is known not to preserve any $\mathrm{BMO}$-like regularity.

\subsection{Strategy of the proof}

Let us briefly explain the strategy behind the proof of our main results. 
We can divide our approach in three fundamental parts.

The first part is the study of the regularity of the velocity. 
As it is well-known, even for a bounded vorticity the associated velocity is in general not Lipschitz, but just $\log$-Lipschitz. 
In the case the vorticity satisfies $\omega\in L^1(\Omega)\cap L^p_{\ul}(\Omega)$ for some $p\in(2,+\infty)$ (actually, it is enough to assume $\omega\in L^q(\Omega)\cap L^p_{\ul}(\Omega)$ for any $1\le q<2<p<+\infty$, see \cref{res:holder_bounded} below), we prove that the velocity satisfies
\begin{equation}
\label{eq:miot}
|v(x)-v(y)|
\lesssim
\max\set*{1,\tfrac1{p-2}}
\,
(\|\omega\|_{L^1(\Omega)}
+
\|\omega\|_{L^p_{\ul}(\Omega)})
\,
p
\,
\d(x,y)^{1-\sfrac2p}
\end{equation}
for all $x,y\in\Omega$.

The H\"older continuity in~\eqref{eq:miot} should not come as a surprise.
Indeed, inequality~\eqref{eq:miot} is a well-known consequence of the Calder\'on--Zygmund theory and the Morrey inequality in the case of the Biot--Savart kernel, see~\cite{Y95}*{Section~4} and~\cite{M16}*{Lemma~2.2 and Remark~2.3} for instance.
Our approach, however, is different, since our proof of~\eqref{eq:miot} solely exploits the metric properties of the kernel (see \cref{ass:est} below) and some elementary integral estimates (known in the literature for bounded vorticities, see the proofs of~\cite{MB02}*{Lemma 8.1} and of~\cite{MP94}*{Lemma~3.1} for example).

The next key idea is the following simple but crucial observation. 
If $\omega\in L^1(\Omega)\cap Y^\Theta_{\ul}(\Omega)$, then~\eqref{eq:miot} holds for any $p\ge3$ and can be rewritten as
\begin{equation}
\label{eq:miot_Theta}
|v(x)-v(y)|
\lesssim
(\|\omega\|_{L^1(\Omega)}
+
\|\omega\|_{Y^\Theta_{\ul}(\Omega)})
\,
\Theta(p)
\,
p
\,
\d(x,y)^{1-\sfrac2p}
\end{equation}
for all $x,y\in\Omega$.
Here and in the rest of the paper, for simplicity and clearly without loss of generality, we can assume that $\Theta(3)\ge1$.
In particular, if $\d(x,y)$ is sufficiently small, then we can take
\begin{equation*}
p=1-\log\d(x,y)
\end{equation*}
in~\eqref{eq:miot_Theta} and discover that
\begin{equation*}
|v(x)-v(y)|
\lesssim
(\|\omega\|_{L^1(\Omega)}
+
\|\omega\|_{Y^\Theta_{\ul}(\Omega)})
\,
\phi_\Theta(\d(x,y))
\end{equation*}
for all $x,y\in\Omega$, where $\phi_\Theta$ is the  function defined in~\eqref{eq:def_phi_Theta}.
In particular, if $\omega\in L^1(\Omega)\cap L^\infty(\Omega)$, then $\Theta$ is bounded and the definition in~\eqref{eq:def_phi_Theta} gives
\begin{equation*}
|v(x)-v(y)|
\lesssim
(\|\omega\|_{L^1(\Omega)}
+
\|\omega\|_{Y^\Theta_{\ul}(\Omega)})
\,
\ell(\d(x,y))
\end{equation*}
for all $x,y\in\Omega$, where $\ell\colon[0,+\infty)\to[0,1]$ is defined as $\ell(0)=0$ and
\begin{equation}
\label{eq:def_ell}
\ell(r)
=
\begin{cases}
r\,(1-\log r)
&
\text{for}\ r\in(0,1],\\
1
& 
\text{for}\ r>1,
\end{cases}
\end{equation}  
recovering the classical $\log$-Lipschitz continuity of the velocity.

The second part is the existence of weak solutions. 
The key tool we use in this part is a simplified version of the celebrated Aubin--Lions Lemma, see \cref{res:aubin-lions} in \cref{sec:aubin-lions}, whose elementary proof is just a combination of the Dunford--Pettis Theorem and the Arzel\`a--Ascoli Compactness Theorem.
With this compactness criterion at hand, we first prove existence of weak solutions of~\eqref{eq:Euler} with vorticity in $L^1\cap L^\infty$.
Having in mind to deal with a general operator~$K$ which may not be necessarily the Biot--Savart one, we cannot rely on the existence theory for smooth solutions, but rather we build a weak solution of~\eqref{eq:Euler} from scratch via a time-stepping argument (a procedure which may be of some interest by itself even in the case of the Biot--Savart law).
The construction of weak solutions with vorticity in $L^1\cap L^p_{\ul}$ then follows by applying the Aubin--Lions-like Lemma to the sequence of bounded weak solutions starting from the truncations of the initial vorticity.

The third and last part is the uniqueness of weak solutions.
In contrast to Yudovich's original approach~\cite{Y95}, we do not employ an energy method by estimating the relative energy between two solutions, but we rather compare the flows associated to the two velocities by an elementary (non-linear) Picard--Lindel\"of iteration-like argument (similar to the one used for bounded vorticities in~\cite{MP94}*{Section~2.3} and in~\cite{L06}), which can also be seen as an estimate for the Wasserstein distance between the two vorticities. It is precisely at this point that the we exploit the \emph{Osgood property}
\begin{equation}
\label{eq:Osgood}
\int_0^{\e^{-2}}\frac{dr}{\phi_\Theta(r)}
=
\int_3^{+\infty}\frac{dp}{p\,\Theta(p)}=+\infty
\end{equation}
and the concavity of the modulus of continuity $\phi_\Theta$ given in~\eqref{eq:def_phi_Theta}. This approach is also somewhat reminiscent of the one by Serfati~\cites{S95A,S95B}, see also~\cite{AKLL15}.

\subsection{Organization of the paper}
The paper is organized as follows.
In \cref{sec:kernel_mapping_props}, we study the mapping properties of the operator~$K$ under some minimal assumptions on the kernel.
In \cref{sec:existence}, we prove the existence of weak solutions, namely  \cref{res:main_existence}, see \cref{res:existence_p_ul} and \cref{res:existence_Theta}.
In \cref{sec:uniqueness}, we establish the uniqueness of weak solutions, namely \cref{res:main_uniqueness}.
Finally, in \cref{sec:aubin-lions}, we state and prove the simplified version of the Aubin--Lions Lemma we need in the existence part, see \cref{res:aubin-lions}.

\section{Mapping properties of the kernel}
\label{sec:kernel_mapping_props}

In this section, we study the mapping properties of the operator~$K$. 
Here and in the rest of the paper, we rely on the metric properties of the underlying kernel in \cref{ass:est} below, and not on its specific form.
These properties are satisfied by the standard Biot--Savart kernel in any (bounded or unbounded) sufficiently smooth domain and on the $2$-dimensional torus (for instance see the aforementioned~\cites{MP84,MP94}).

\begin{assumption}[Estimates on the kernel]\label{ass:est} 
We assume that the kernel $k\colon\Omega\times\Omega\to\mathbb R^2$ satisfies
\begin{equation}
\label{eq:k_kernel_bound}
|k(x,y)|\le\frac{C_1}{\d(x,y)}
\qquad
\forall x,y\in\Omega,\ x\ne y,
\end{equation}
and
\begin{equation}
\label{eq:k_kernel_oscillation_bound}
|k(x,z)-k(y,z)|
\le
C_2\,
\frac{\d(x,y)}{\d(x,z)\,\d(y,z)}
\qquad
\forall x,y,z\in\Omega,\ z\ne x,y,
\end{equation}
for some constants $C_1, C_2>0$. 
\end{assumption}

We begin by establishing the H\"older continuity of the velocity, extending to our setting the proof of~\cite{MB02}*{Lemma 8.1} and of~\cite{MP94}*{Lemma~3.1}.

\begin{theorem}[H\"older continuity]
\label{res:holder_bounded}
Let \cref{ass:est} be in force and let $q\in[1,2)$ and $p\in(2,+\infty)$.
If $\omega\in L^q(\Omega)\cap L^p_{\ul}(\Omega)$, then the function 
\begin{equation}
\label{eq:def_Kw}
K\omega(x)
=
\int_\Omega k(x,z)\,\omega(z)\di z,
\quad
x\in\Omega,
\end{equation}
is well defined and satisfies
$K\omega\in C_b^{0,1-\sfrac 2p}(\Omega;\R^2)$
with
\begin{equation}
\label{eq:Kw_bounded}
\|K\omega\|_{L^\infty(\Omega;\,\R^2)}
\lesssim
\max\set*{1,\tfrac1{p-2}}
\,
(
\|\omega\|_{L^q(\Omega)}
+
\|\omega\|_{L^p_{\ul}(\Omega)}
)
\end{equation}
and
\begin{equation}
\label{eq:Kw_holder}
|K\omega(x)-K\omega(y)|
\lesssim
\max\set*{1,\tfrac1{p-2}}
\,
(\|\omega\|_{L^q(\Omega)}
+
\|\omega\|_{L^p_{\ul}(\Omega)})
\,
p\,\d(x,y)^{1-\sfrac 2p}
\end{equation}
for all $x,y\in\Omega$.
The implicit constants in~\eqref{eq:Kw_bounded} and~\eqref{eq:Kw_holder} only depend on the constants~$C_1$ and~$C_2$ in \cref{ass:est} and on the exponent~$q$ (but not on the exponent~$p$). 
\end{theorem}

\begin{remark}
Observe that the H\"older continuity of order $1-\sfrac{2}{p}$ is the same that would follow by using Morrey's inequality from the $W^{1,p}$ Sobolev regularity of the velocity field associated (via the standard Biot--Savart law) to an $L^p$ vorticity. In the proof below, we make no use of such tools, which are not available in the case of a kernel satisfying \cref{ass:est} only. 
\end{remark}

\begin{proof}[Proof of \cref{res:holder_bounded}] 
We divide the proof in three steps. 

\smallskip

\textit{Step~1: proof of~\eqref{eq:Kw_bounded}}.
Let $x\in\Omega$ be fixed.
We start by noticing that the function in~\eqref{eq:def_Kw} can be estimated as
\begin{align*}
|K\omega(x)|
\le
\int_{\Omega \cap B_1(x)}
|k(x,z)|\,|\omega(z)|\di z
+
\int_{\Omega\setminus B_1(x)}
|k(x,z)|\,|\omega(z)|\di z.
\end{align*}	
On the one side, by~\eqref{eq:k_kernel_bound} we can estimate
\begin{align*}
\int_{\Omega \cap B_1(x)}
|k(x,z)|\,|\omega(z)|\di z
&\le
C_1
\int_{\Omega \cap B_1(x)}
\frac{|\omega(z)|}{\d(x,z)}\di z
\lesssim
\|\omega\|_{L^p(\Omega\cap B_1(x))}
\left(\int_0^1 r^{1-p'}\di r\right)^{\sfrac1{p'}}
\\
&\lesssim 
\|\omega\|_{L^p_{\ul}(\Omega)}
\left(\frac{1}{2-p'}\right)^{\sfrac 1{p'}}
\lesssim
\max\set*{1,\tfrac1{p-2}}\,\|\omega\|_{L^p_{\ul}(\Omega)},
\end{align*}
where $p' = \sfrac{p}{(p-1)}\in(1,2)$.
On the other side, again by~\eqref{eq:k_kernel_bound}, we can estimate
\begin{align*}
\int_{\Omega\setminus B_1(x)}
|k(x,z)|\,|\omega(z)|\di z
\le
C_1
\int_{\Omega\setminus B_1(x)}
\frac{|\omega(z)|}{\d(x,z)}\di z
\lesssim
\|\omega\|_{L^q(\Omega)}.
\end{align*}
In conclusion, we find that
\begin{align*}
|K\omega(x)|
\lesssim
\|\omega\|_{L^q(\Omega)}
+
\max\set*{1,\tfrac1{p-2}}\,\|\omega\|_{L^p_{\ul}(\Omega)}\end{align*}
for each $x\in\Omega$, proving~\eqref{eq:Kw_bounded}.

\smallskip

\textit{Step~2: proof of~\eqref{eq:Kw_holder}, part 1}.
Let $x,y\in\Omega$ be fixed and assume that $d=\d(x,y)<1$. 
We note that
\begin{align}
|K\omega(x)-K\omega(y)|
&\le
\int_\Omega
|k(x,z)-k(y,z)|\,|\omega(z)|\di z 
\nonumber\\
&=
\label{eq:k_kernel_split_int}
\left(
\int_{\Omega\setminus B_2(x)}
\! + 
\int_{\Omega \cap (B_2(x)\setminus B_{2d}(x))}
\! +
\int_{\Omega \cap B_{2d}(x)}
\right)
|k(x,z)-k(y,z)|\,|\omega(z)|\di z.
\end{align}
By~\eqref{eq:k_kernel_oscillation_bound}, we can estimate the first integral in~\eqref{eq:k_kernel_split_int} as
\begin{align*}
\int_{\Omega\setminus B_2(x)}
|k(x,z)-k(y,z)|\,|\omega(z)|\di z
\le
C_2\,
\d(x,y)
\int_{\Omega\setminus B_2(x)}
\frac{|\omega(z)|}{\d(x,z)\,\d(y,z)}\di z
\lesssim
\d(x,y)\,
\|\omega\|_{L^q(\Omega)}.
\end{align*}
Again by~\eqref{eq:k_kernel_oscillation_bound}, we can estimate the second integral in~\eqref{eq:k_kernel_split_int} as
\begin{align*}
\int_{\Omega \cap (B_2(x)\setminus B_{2d}(x))}
|k(x,z)-k(y,z)|&\,|\omega(z)|\di z
\le
C_2\,
\d(x,y)
\int_{\Omega \cap (B_2(x)\setminus B_{2d}(x))}
\frac{|\omega(z)|}{\d(x,z)\,\d(y,z)}\di z.
\end{align*} 
Since $\d(x,y)=d$ and $\d(x,z)\geq 2d$, we have
\begin{equation*}
\d(x,z) \leq \d(x,y) + \d(y,z) = d + \d(y,z) \leq \tfrac{1}{2}\,\d(x,z) + \d(y,z),
\end{equation*}
and therefore
\begin{equation*}
\d(y,z)\ge\tfrac{1}{2}\,\d(x,z)
\quad
\text{for all}\ z\in \Omega\setminus B_{2d}(x).
\end{equation*}
Hence, we can estimate
\begin{align*}
\int_{\Omega \cap(B_2(x)\setminus B_{2d}(x))}
\frac{|\omega(z)|}{\d(x,z)\,\d(y,z)}\di z
&\lesssim
\int_{\Omega \cap (B_2(x)\setminus B_{2d}(x))}
\frac{|\omega(z)|}{\d(x,z)^2}\di z.
\end{align*}
Finally, using~\eqref{eq:k_kernel_bound} and observing that $B_{2d}(x)\subset B_{3d}(y)$, we can estimate the third integral in~\eqref{eq:k_kernel_split_int} as
\begin{align*}
\int_{\Omega \cap B_{2d}(x)}
|k(x,z)-k(y,z)|\,|\omega(z)|\di z
&\lesssim
\int_{\Omega \cap B_{2d}(x)}
\frac{|\omega(z)|}{\d(x,z)}\di z
+
\int_{\Omega \cap B_{3d}(y)}
\frac{|\omega(z)|}{\d(y,z)}\di z.
\end{align*}

\smallskip

\textit{Step~3: proof of~\eqref{eq:Kw_holder}, part 2}.
To conclude, we just need to estimate the functions
\begin{equation*}
\alpha(d)
=
\sup_{x\in\Omega}
\int_{\Omega \cap (B_2(x)\setminus B_{2d}(x))}
\frac{|\omega(z)|}{\d(x,z)^2}\di z
\quad\text{and}\quad
\beta(d)
=
\sup_{x\in\Omega}
\int_{\Omega \cap B_{3d}(x)}
\frac{|\omega(z)|}{\d(x,z)}\di z
\end{equation*}
defined for $d\in(0,1]$. 
Concerning the function~$\alpha$, by H\"older's inequality we can estimate
\begin{equation}
\label{eq:alpha_estimate}
\begin{split}
\alpha(d)
&\lesssim
\left(
\sup_{x\in\Omega}\|\omega\|_{L^p(\Omega \cap B_2(x))}
\right)
\left(
\int_{2d}^2 r^{1-2p'}\di r
\right)^{\sfrac1{p'}}
\\
&\lesssim
\|\omega\|_{L^p_{\ul}(\Omega)} \,
\left( \frac{2^{2-2p'}}{2p'-2}\right)^{\sfrac{1}{p'}} \,
\left( d^{2-2p'}-1 \right)^{\sfrac{1}{p'}}
\lesssim
\|\omega\|_{L^p_{\ul}(\Omega)} 
\,
p
\,
d^{\,-\sfrac 2p}.
\end{split}
\end{equation}
We can argue similarly for the function~$\beta$, obtaining
\begin{equation}
\label{eq:beta_estimate}
\begin{split}
\beta(d)
& \lesssim
\left(
\sup_{x\in\Omega}\|\omega\|_{L^p(\Omega \cap B_3(x))}
\right)
\left( \int_{0}^{3d} r^{1-p'}\di r \right)^{\sfrac1{p'}} 
\\
& \lesssim
\|\omega\|_{L^p_{\ul}(\Omega)} 
\,
\left( \frac{3^{2-p'}}{2-p'} \right)^{\sfrac{1}{p'}} \, d^{\,\sfrac{(2-p')}{p'}} 
\lesssim
\|\omega\|_{L^p_{\ul}(\Omega)} 
\,\frac{p}{p-2}
\, d^{\,1-\sfrac{2}{p}}.
\end{split}
\end{equation}
Recalling the bound~\eqref{eq:Kw_bounded}, this is enough to conclude the proof of~\eqref{eq:Kw_holder}.
\end{proof}

From \cref{res:holder_bounded} we easily derive the following result, generalizing the well-known $\log$-Lipschitz continuity of the velocity valid for vorticities in $L^1\cap L^\infty$.

\begin{corollary}[$\phi_\Theta$-continuity]
\label{res:almost_Lip}
Let \cref{ass:est} be in force and let $q\in[1,2)$.
If $\omega\in L^q(\Omega)\cap Y^\Theta_{\ul}(\Omega)$, then 
$K\omega\in C_b^{0,\phi_\Theta}(\Omega;\R^2)$
with
\begin{equation}
\label{eq:Kw_bounded_Theta}
\|K\omega\|_{L^\infty(\Omega;\,\R^2)}
\lesssim
\|\omega\|_{L^q(\Omega)}
+
\|\omega\|_{Y^\Theta_{\ul}(\Omega)}
\end{equation}
and
\begin{equation}
\label{eq:Kw_almost_Lip}
|K\omega(x)-K\omega(y)|
\lesssim
(\|\omega\|_{L^q(\Omega)}
+
\|\omega\|_{Y^\Theta_{\ul}(\Omega)})
\,
\phi_\Theta(\d(x,y))
\end{equation}
for all $x,y\in\Omega$, where $\phi_\Theta$ is the function defined in~\eqref{eq:def_phi_Theta}.
The implicit constants in~\eqref{eq:Kw_bounded_Theta} and~\eqref{eq:Kw_almost_Lip} only depend on the constants $C_1$ and $C_2$ in \cref{ass:est} and on the exponent~$q$ (but not on the behavior of the growth function~$\Theta$ at infinity). 
\end{corollary}

\begin{proof}
We divide the proof in two steps.

\smallskip

\textit{Step~1: proof of~\eqref{eq:Kw_bounded_Theta}}.
Taking $p=3$ in~\eqref{eq:Kw_bounded}, since $\Theta(3)\ge1$ by assumption, we immediately see that
\begin{align*}
\|K\omega\|_{L^\infty(\Omega;\,\R^2)}
\lesssim
\|\omega\|_{L^q(\Omega)}
+
\|\omega\|_{L^3_{\ul}(\Omega)}
\lesssim
\|\omega\|_{L^q(\Omega)}
+
\Theta(3)\,\|\omega\|_{Y^\Theta_{\ul}(\Omega)}
\lesssim
\|\omega\|_{L^q(\Omega)}
+
\|\omega\|_{Y^\Theta_{\ul}(\Omega)}.
\end{align*} 

\smallskip

\textit{Step~2: proof of~\eqref{eq:Kw_almost_Lip}}.
Let $x,y\in\Omega$ be such that $d=\d(x,y)\in(0,e^{-2}]$. 
Taking $p=1-\log d\in[3,+\infty)$ in~\eqref{eq:Kw_holder}, we get that $\Theta(1-\log d)\ge\Theta(3)\ge1$ and thus
\begin{align*}
|K\omega(x)-K\omega(y)|
&\lesssim
\big(\|\omega\|_{L^q(\Omega)}
+
\Theta(1-\log d)\|\omega\|_{Y^\Theta_{\ul}(\Omega)}\big)
\,
(1-\log d)
\,
d^{1-\frac2{1-\log d}}
\\
&\le
(\|\omega\|_{L^q(\Omega)}
+
\|\omega\|_{Y^\Theta_{\ul}(\Omega)})
\,
\Theta(1-\log d)
\,
(1-\log d)
\,
d^{1-\frac2{1-\log d}}
\\
&\lesssim
(\|\omega\|_{L^q(\Omega)}
+
\|\omega\|_{Y^\Theta_{\ul}(\Omega)})
\,
d
\,
(1-\log d)
\,
\Theta(1-\log d).
\end{align*}
Thanks to the bound~\eqref{eq:Kw_bounded_Theta}, this proves~\eqref{eq:Kw_almost_Lip}.
\end{proof}

\begin{remark}[Stronger versions of~\eqref{eq:Kw_holder} and~\eqref{eq:Kw_almost_Lip}]
\label{rem:K_stronger}
For later use, we observe that, in Steps~2 and~3 of the proof of \cref{res:holder_bounded}, we actually showed that
\begin{equation}\label{eq:holder_strong}
\int_\Omega
|k(x,z)-k(y,z)|\,|\omega(z)|\di z
\lesssim
\max\set*{1,\tfrac1{p-2}}
\,
(\|\omega\|_{L^q(\Omega)}
+
\|\omega\|_{L^p_{\ul}(\Omega)})
\,
p\,\d(x,y)^{1-\sfrac 2p}
\end{equation}
for all $x,y\in\Omega$, where the implicit constant only depends on $C_1$ and $C_2$ in \cref{ass:est} and on~$q$ (but not on~$p$). 
Consequently, in Step~2 of the proof of \cref{res:almost_Lip}, we actually showed that
\begin{equation}
\label{eq:Kw_almost_Lip_strong}
\int_\Omega
|k(x,z)-k(y,z)|\,|\omega(z)|\di z
\lesssim
(\|\omega\|_{L^q(\Omega)}
+
\|\omega\|_{Y^\Theta_{\ul}(\Omega)})
\,
\phi_\Theta(\d(x,y))
\end{equation}
for all $x,y\in\Omega$, where the implicit constant only depends on the constants $C_1$ and $C_2$ in \cref{ass:est} and on the exponent~$q$ (but not on the behavior of the growth function~$\Theta$ at infinity).
\end{remark}

\begin{remark}[Yudovich's approach]
\label{rem:Yudovich_approach}
Inequality~\eqref{eq:Kw_almost_Lip} in \cref{res:almost_Lip} can also be obtained by re-doing the estimates~\eqref{eq:alpha_estimate} and~\eqref{eq:beta_estimate} following Yudovich's approach in~\cite{Y95}*{Lemma~3.1}.
Indeed, for $\eps\in(0,\sfrac{1}{3})$ we have
\begin{equation*}
\begin{split}
\alpha(d)
&\lesssim
d^{-2\eps}
\int_{\Omega\cap(B_2(x)\setminus B_{2d}(x))}\frac{|\omega(z)|}{\d(x,z)^{2(1-\eps)}}\di z
\lesssim
d^{-2\eps}
\left(
\sup_{x\in\Omega}\|\omega\|_{L^{\sfrac1\eps}(\Omega \cap B_2(x))}
\right)
\left(
\int_{2d}^2 r^{-1}\di r
\right)^{1-\eps}
\\
&\lesssim
\|\omega\|_{Y^\Theta_{\ul}(\Omega)} 
\,
\Theta(\tfrac1\eps)
\,
(1-\log d)^{1-\eps}
\,
d^{-2\eps}
\end{split}	
\end{equation*}  
by applying H\"older's inequality with exponents $\sfrac1\eps$ and $\sfrac1{(1-\eps)}$. 
A similar computation shows that
\begin{equation*}
\beta(d)
\lesssim
\left(
\sup_{x\in\Omega}\|\omega\|_{L^{\sfrac1\eps}(\Omega \cap B_{3d}(x))}
\right)
\left(
\int_0^{3d} r^{1-\sfrac{1}{(1-\eps)}}\di r
\right)^{1-\eps}
\lesssim
\|\omega\|_{Y^\Theta_{\ul}(\Omega)} 
\,
\Theta(\tfrac1\eps)
\,
d^{1-2\eps},
\end{equation*}  
so that
\begin{equation*}
|K\omega(x)-K\omega(y)|
\lesssim
(\|\omega\|_{L^q(\Omega)}
+
\|\omega\|_{Y^\Theta_{\ul}(\Omega)})
\,
\tilde\psi_\Theta(\d(x,y))
\end{equation*}
for all $x,y\in\Omega$, where
\begin{equation}
\label{eq:def_psi_Theta_Yudovich}
\tilde\psi_\Theta(d)
=
\inf
\set*{
\Theta(\tfrac1\eps)
\,
(1-\log d)^{1-\eps}
\,
d^{1-2\eps}
:
0<\eps<\tfrac13
}
\end{equation}
for all $d\in(0,e^{-2}]$, in analogy with the definition in~\eqref{eq:yudo_inf_function}.
Due to its implicit definition in~\eqref{eq:def_psi_Theta_Yudovich}, the function $\tilde\psi_\Theta$ is not easily exploitable for further computations (at least, unless $\Theta$ has a more explicit expression, such as~\eqref{eq:yudo_Theta_example}). 
However, as in the proof of \cref{res:almost_Lip}, one realizes that the choice $\eps=\sfrac1{(1-\log d)}$ in~\eqref{eq:def_psi_Theta_Yudovich} gives
\begin{align*}
\tilde\psi_\Theta(d)
\lesssim
d\,(1-\log d)\,\Theta(1-\log d)
\lesssim
\phi_\Theta(d)
\end{align*}
for all $d\in(0,e^{-2}]$, so that we recover~\eqref{eq:Kw_almost_Lip} also via this alternative approach. 
\end{remark}

\section{Existence of weak solutions (\texorpdfstring{\cref{res:main_existence}}{Theorem 1.8})}
\label{sec:existence}

In this section, we show existence of weak solutions for the Euler equations~\eqref{eq:Euler}. Here and in the rest of the paper, in addition to \cref{ass:est}, we assume two further properties concerning the divergence and the behavior at the boundary of the velocity generated by the operator~$K$.

\begin{assumption}[Bounded divergence and no-flow boundary condition]
\label{ass:diverg_non-tang}
Let $p\in(2,+\infty]$ be given.
We assume that the operator
\begin{equation*}
K\colon L^1(\Omega)\cap L^p_{\ul}(\Omega)\to C^{0,1-\sfrac 2p}_b(\Omega;\mathbb R^2) 
\end{equation*}
defined in~\eqref{eq:def_Kw} of \cref{res:holder_bounded} is such that the distributional divergence $\div(K\omega)$ satisfies 
\begin{equation}
\label{eq:null_div_K}
\|\div(K\omega)\|_{L^\infty(\Omega)}
\le 
C_3\,\|\omega\|_{L^1(\Omega)}
\end{equation}
for all $\omega\in L^1(\Omega)\cap L^p_{\ul}(\Omega)$, for some constant $C_3>0$.
If $\Omega\subset\R^2$ is an open set with sufficiently regular boundary, we assume the {\em no-flow boundary condition}
\begin{equation}
\label{eq:non-tang_K}
\nu_\Omega\cdot K\omega=0
\quad
\text{on}\ \de\Omega	
\end{equation}
for all $\omega\in L^1(\Omega)\cap L^p_{\ul}(\Omega)$.
Condition~\eqref{eq:non-tang_K} is empty if either $\Omega=\R^2$ or $\Omega=\T^2$.
\end{assumption}

Note that \cref{ass:diverg_non-tang} is trivially satisfied in the case of the standard Biot--Savart law, since the specific form of the kernel entails $\div(K\omega)=0$.

\medskip

We will employ the following standard definition of \emph{weak solution} and of \emph{Lagrangian weak solution} of the Euler equations~\eqref{eq:Euler}.

\begin{definition}[Weak solution]
\label{def:weak_solution}
Let $p\in(2,+\infty]$.
Given an initial condition for the vorticity $\omega_0\in L^1(\Omega)\cap L^p_{\ul}(\Omega)$, we say that the couple $(\omega,v)$ is a \emph{weak solution of~\eqref{eq:Euler} with vorticity in $L^1\cap L^p_{\ul}$} provided that:
\begin{enumerate}[(i)]

\item 
$\omega\in L^\infty_{\loc}([0,+\infty);L^1(\Omega)\cap L^p_{\ul}(\Omega))$;

\item 
$v=K\omega$ in $L^\infty_{\loc}([0,+\infty);C_b(\Omega;\R^2))$;

\item 
given $T\in(0,+\infty)$, for all $\phi\in C^1_c([0,T] \times \Omega)$ it holds
\begin{equation*}
\int_\Omega\phi(T,x)\,\omega(T,x)\di x
-
\int_\Omega\phi(0,x)\,\omega_0(x)\di x
=
\int_0^T\int_\Omega\omega\,(\de_t\phi+v\cdot\nabla\phi)\di x\di t.
\end{equation*}
\end{enumerate} 
A weak solution $(\omega,v)$ of~\eqref{eq:Euler} with vorticity in $L^1\cap L^p_{\ul}$ is called \emph{Lagrangian} if $\omega(t,\cdot)=X(t,\cdot)_\#\omega_0$ for a.e.\ $t\in[0,+\infty)$, where $X$ is a flow associated to the velocity field $v$. 
\end{definition}

In \cref{def:weak_solution}, we say that $X$ is a flow associated to the velocity field $v$ if
\begin{equation}
\label{eq:ode_peano}
\begin{cases}
\frac{\di}{\di t}X(t,x)=v(t,X(t,x))
&
\text{for}\ (t,x)\in(0,+\infty)\times\Omega,
\\[3mm]
X(0,x)=x
&
\text{for}\
x\in\Omega.
\end{cases}
\end{equation}
The ODE in~\eqref{eq:ode_peano} is understood in the classical sense. Since the velocity belongs to the space $L^\infty_{\loc}([0,+\infty);C_b(\Omega;\R^2))$ and satisfies the no-flow boundary condition~\eqref{eq:non-tang_K}, the existence of a solution~$X$ of the problem~\eqref{eq:ode_peano} follows from the Peano  Theorem.
The relation $\omega(t,\cdot)=X(t,\cdot)_\#\omega_0$ stands for the usual \emph{push-forward}, i.e.\ 
\begin{equation*}
\int_\Omega\omega(t,\cdot)\,\phi\di x
=
\int_\Omega\omega_0\,\phi(X(t,\cdot)) \di x
\end{equation*}
for all bounded measurable functions $\phi\colon\Omega\to\R$. 

\medskip

We are now ready to deal with the existence of weak solutions. We begin with the case of weak solutions with vorticity in $L^1\cap L^\infty$. The result in \cref{res:existence_bounded} below is well known in the case of the standard Biot--Savart kernel. In our more general setting, we cannot rely on any general results of existence of smooth solutions for smooth data, due to the lack of an evolution equation for the velocity. Instead, we construct the solution by combining a time-stepping argument with the Aubin--Lions-like Lemma given in \cref{sec:aubin-lions}.

Here and in the following, $\ell\colon[0,+\infty)\to[0,1]$ denotes the log-Lipschitz modulus of continuity defined in~\eqref{eq:def_ell}.

\begin{theorem}[Existence in $L^1\cap L^\infty$]
\label{res:existence_bounded} 
Let Assumptions~\ref{ass:est} and~\ref{ass:diverg_non-tang} be in force.
Then there exists a Lagrangian weak solution $(\omega,v)$ of~\eqref{eq:Euler} with vorticity in $L^1\cap L^\infty$ starting from the initial datum $\omega_0\in L^1(\Omega)\cap L^\infty(\Omega)$
such that
\begin{equation}
\label{eq:existence_bounded_L1}
\|\omega\|_{L^\infty([0,T];\,L^1(\Omega))}
\le
\|\omega_0\|_{L^1(\Omega)},
\end{equation}
\begin{equation}
\label{eq:existence_bounded_Linfty}
\|\omega\|_{L^\infty([0,T];\,L^\infty(\Omega))}
\le
\exp(C_3T\|\omega_0\|_{L^1(\Omega)})
\,
\|\omega_0\|_{L^\infty(\Omega)},
\end{equation} 
\begin{equation}
\label{eq:existence_bounded_v_bounded}
\|v\|_{L^\infty([0,T];\,L^\infty(\Omega;\,\R^2))}
\lesssim
\|\omega_0\|_{L^1(\Omega)}
+
\|\omega_0\|_{L^\infty(\Omega)},
\end{equation}
and
\begin{equation}
\label{eq:existence_bounded_v_log-Lip}
|v(t,x)-v(t,y)|\lesssim
(\|\omega_0\|_{L^1(\Omega)}
+
\|\omega_0\|_{L^\infty(\Omega)})\,
\ell(\d(x,y)),
\;\;
\text{for all $x,y\in\Omega$ and a.e.~$t\in [0,T]$,}
\end{equation}
for all $T\in(0,+\infty)$, where the implicit constants may depend on the chosen~$T$. \end{theorem}

\begin{proof}
Let $T\in(0,+\infty)$ and $\omega_0\in L^1(\Omega)\cap L^\infty(\Omega)$ be fixed and define $v_0=K\omega_0$.

\smallskip

\textit{Step~1: construction of $(\omega^n,v^n)_{n\in\N}$ by time-stepping}.
Let $n\in\N$ and consider the time step~$\sfrac{T}{n}$.
We construct a sequence of functions $(\omega^n,v^n)_{n\in\N}$ as follows.
We set $\omega^n_0=\omega_0$ for all $n\in\N$ by definition.
If $t\in[(j-1)\sfrac{T}{n},\sfrac{jT}{n}]$ for some $j\in\set*{1,\dots,n}$, then we define $\omega^n(t,\cdot)=w(t,\cdot)$, where $w$ is advected on the interval $[(j-1)\sfrac{T}{n},\sfrac{jT}{n}]$ by the time-independent velocity 
$v^n((j-1)\sfrac{T}{n},\cdot)=K\omega^n((j-1)\sfrac{T}{n},\cdot)$,
that is, $w$ solves
\begin{equation}
\label{eq:bingo}
\begin{cases}
\de_t w+\div(v^n((j-1)\sfrac{T}{n},\cdot) \ w)=0 
& \text{in}\ ((j-1)\sfrac{T}{n},\sfrac{jT}{n})\times\Omega,\\[2mm]
w((j-1)\sfrac{T}{n},\cdot)=\omega^n((j-1)\sfrac{T}{n},\cdot)
& \text{on}\ \Omega,
\end{cases}
\end{equation}
in the distributional sense.

We show that the couple $(\omega^n,v^n)$ is well defined for each $n\in\N$ by an inductive argument. 
By \cref{res:almost_Lip} for $t\in[(j-1)\sfrac{T}{n},\sfrac{jT}{n}]$ and $j=1,\dots,n$ we have
\begin{equation}
\label{eq:v_n_bounded}
\|v^n(t,\cdot)\|_{L^\infty(\Omega;\,\R^2)}
\lesssim
\|\omega^n((j-1)\sfrac{T}{n},\cdot)\|_{L^1(\Omega)}
+
\|\omega^n((j-1)\sfrac{T}{n},\cdot)\|_{L^\infty(\Omega)}
\end{equation}
and
\begin{equation}
\label{eq:v_n_log-Lip}
|v^n(t,x)-v^n(t,y)|
\lesssim
\big(\|\omega^n((j-1)\sfrac{T}{n},\cdot)\|_{L^1(\Omega)}
+
\|\omega^n((j-1)\sfrac{T}{n},\cdot)\|_{L^\infty(\Omega)}\big)
\,
\ell(\d(x,y)),
\;
\forall x,y\in\Omega.
\end{equation}
We argue inductively on $j=1,\dots,n$.
For $j=1$, we have $\omega^n(t,\cdot)=X^n(t,\cdot)_\#\omega_0$ for all $t\in[0,\sfrac{T}{n}]$, where $[0,\sfrac{T}{n}]\ni t\mapsto X^n(t,\cdot)$ is the flow associated to the time-independent velocity field $v_0^n=v_0$. 
Consequently, for all $t\in[0,\sfrac{T}{n}]$ we can estimate
\begin{equation*}
\|\omega^n(t,\cdot)\|_{L^1(\Omega)}
\le
\|\omega^n_0\|_{L^1(\Omega)}
=
\|\omega_0\|_{L^1(\Omega)}
\end{equation*}
and, thanks to~\eqref{eq:null_div_K} in \cref{ass:diverg_non-tang}, 
\begin{equation*}
\begin{split}
\|\omega^n(t,\cdot)\|_{L^\infty(\Omega)}
&\le
\exp\left(\int_0^t\|\div v^n(s,\cdot)\|_{L^\infty(\Omega)}\di s\right)
\,
\|\omega^n_0\|_{L^\infty(\Omega)}
\\
&\le
\exp\left(\tfrac Tn\|\div (K\omega_0)\|_{L^\infty(\Omega)}\right)
\,
\|\omega_0\|_{L^\infty(\Omega)}
\le
\exp\left(\tfrac{C_3T}n\|\omega_0\|_{L^1(\Omega)}\right)
\,
\|\omega_0\|_{L^\infty(\Omega)}.
\end{split}
\end{equation*}
Now, for $j\in\set*{2,\dots,n-1}$, let us assume that 
\begin{equation*}
\|\omega^n(\sfrac{jT}{n},\cdot)\|_{L^1(\Omega)}
\le
\|\omega_0\|_{L^1(\Omega)}
\end{equation*}
and 
\begin{equation*}
\|\omega^n(\sfrac{jT}{n},\cdot)\|_{L^\infty(\Omega)}
\le
\exp\left(\tfrac{jC_3T}n\|\omega_0\|_{L^1(\Omega)}\right)\|\omega_0\|_{L^\infty(\Omega)}.
\end{equation*}
Then $\omega^n(t,\cdot)=X^n(t,\cdot)_\#\omega^n(\sfrac{jT}{n},\cdot)$ for all $t\in[\sfrac{jT}{n},(j+1)\sfrac{T}{n}]$, where $[\sfrac{jT}{n},(j+1)\sfrac{T}{n}]\ni t\mapsto X^n(t,\cdot)$ is the flow associated to the time-independent velocity field $v^n(\sfrac{jT}{n},\cdot)=K\omega^n(\sfrac{jT}{n},\cdot)$.
Consequently, we can estimate
\begin{equation*}
\|\omega^n(t,\cdot)\|_{L^1(\Omega)}
\le
\|\omega^n(\sfrac{jT}{n},0)\|_{L^1(\Omega)}
\le
\|\omega_0\|_{L^1(\Omega)}
\end{equation*}
and, thanks to~\eqref{eq:null_div_K} in \cref{ass:diverg_non-tang}, 
\begin{equation*}
\begin{split}
\|\omega^n(t,\cdot)\|_{L^\infty(\Omega)}
&\le
\exp\left(\int_{\sfrac{jT}{n}}^t\|\div v^n(s,\cdot)\|_{L^\infty(\Omega)}\di s\right)
\,
\|\omega^n(\sfrac{jT}{n},\cdot)\|_{L^\infty(\Omega)}
\\
&\le
\exp\left(\tfrac{C_3T}n\|\omega^n(\sfrac{jT}{n},\cdot)\|_{L^1(\Omega)}\right)
\,
\|\omega^n(\sfrac{jT}{n},\cdot)\|_{L^\infty(\Omega)}
\\
&\le
\exp\left(\tfrac{C_3T}n\|\omega_0\|_{L^1(\Omega)}\right)
\,
\exp\left(\tfrac{jC_3T}n\|\omega_0\|_{L^1(\Omega)}\right)\|\omega_0\|_{L^\infty(\Omega)}
\\
&=
\exp\left(\tfrac{(j+1)C_3T}n\|\omega_0\|_{L^1(\Omega)}\right)\|\omega_0\|_{L^\infty(\Omega)}
\end{split}
\end{equation*}
for all $t\in[\sfrac{jT}{n},(j+1)\sfrac{T}{n}]$. 
Therefore, we conclude that 
\begin{equation}
\label{eq:w_n_bounded_L1}
\|\omega^n(t,\cdot)\|_{L^1(\Omega)}
\le
\|\omega_0\|_{L^1(\Omega)}
\end{equation}
and  
\begin{equation}
\label{eq:w_n_bounded_Linfty}
\|\omega^n(t,\cdot)\|_{L^\infty(\Omega)}
\le
\exp\left(C_3T\|\omega_0\|_{L^1(\Omega)}\right)
\,
\|\omega_0\|_{L^\infty(\Omega)}
\end{equation} 
for all $t\in[0,T]$ and $n\in\N$.
In particular, the uniform bounds~\eqref{eq:w_n_bounded_L1} and~\eqref{eq:w_n_bounded_Linfty} in combination with the inequalities~\eqref{eq:v_n_bounded} and~\eqref{eq:v_n_log-Lip} imply that $(v^n)_{n\in\N}$ is uniformly equi-bounded and uniformly equi-continuous (uniformly in time) with modulus of continuity~$\ell$.
Observe that we actually proved that $\omega^n(t,\cdot)=X^n(t,\cdot)_\#\omega_0$ for all $t\in[0,T]$ and $n\in\N$, where $X^n$ is the flow associated to the (piecewise 
constant-in-time) velocity field~$v^n$.
Finally, it is immediate to check that $(\omega^n,v^n)$ solves
\begin{equation}
\label{eq:bongo}
\begin{cases}
\de_t\omega^n+\div(v^n\omega^n)=0
&
\text{in}\
(0,T)\times\Omega,
\\[2mm]
\omega^n|_{t=0}=\omega_0
&
\text{on}\
\Omega,	
\end{cases}
\end{equation}
in the distributional sense for each $n\in\N$.

\smallskip

\textit{Step~2: properties of $(\omega^n)_{n\in\N}$}.
We now claim that the sequence $(\omega^n)_{n\in\N}$ satisfies the hypotheses of \cref{res:aubin-lions}.
Indeed, \eqref{eq:AL_bounded} follows immediately from~\eqref{eq:w_n_bounded_L1}.
By~\eqref{eq:w_n_bounded_Linfty}, we have
\begin{align*}
\sup_{n\in\N}
\|\omega^n\|_{L^\infty([0,T];L^1(A))}
\le
|A|\,
\sup_{n\in\N}
\|\omega^n\|_{L^\infty([0,T];L^\infty(\Omega))}
\le
\exp\left(C_3 T \|\omega_0\|_{L^1(\Omega)}\right)
\,
\|\omega_0\|_{L^\infty(\Omega)}
\,
|A|
\end{align*}
for all $A\subset\Omega$, from which~\eqref{eq:AL_eps-delta} immediately follows. 
Assumption~\eqref{eq:AL_eps-Omega_eps} is empty if $|\Omega|<+\infty$. In order to show~\eqref{eq:AL_eps-Omega_eps} when $|\Omega|=+\infty$, we exploit the representation $\omega^n(t,\cdot)=X^n(t,\cdot)_\#\omega_0$.
Given~$\eps>0$, we choose $r>0$ such that
\begin{equation*}
\int_{\Omega\setminus B_r}|\omega_0|\di x<\eps.
\end{equation*}
Note that, for any $x\in\Omega$, we have
\begin{align*}
\sup_{n\in\N}
\sup_{t\in[0,T]}
\d(X^n(t,x),x)
\le
T\sup_{n\in\N}\|v^n\|_{L^\infty([0,T];\,L^\infty(\Omega;\,\R^2))}
\lesssim
T(\|\omega_0\|_{L^1(\Omega)}
+
\|\omega_0\|_{L^\infty(\Omega)})
\end{align*}
and thus $X^n(t,B_r)\subset B_R$ for all $n\in\N$ and $t\in [0,T]$,
where $R=r+CT$ and $C>0$ is a constant depending only on $\|\omega_0\|_{L^1(\Omega)}$ and $\|\omega_0\|_{L^\infty(\Omega)}$.
Therefore $X^n(t,\cdot)^{-1}(\Omega\setminus B_R) \subset \Omega\setminus B_r$ for all $n\in\N$ and $t\in [0,T]$,
and, consequently, we conclude that
\begin{align*}
\sup_{n\in\N}
\sup_{t\in[0,T]}
\int_{\Omega\setminus B_R}|\omega^n(t,\cdot)|\di x
&\le
\sup_{n\in\N}
\sup_{t\in[0,T]}
\int_{X^n(t,\cdot)^{-1}(\Omega\setminus B_R)}|\omega_0|\di x
\le
\int_{\Omega\setminus B_r}|\omega_0|\di x
<
\eps,
\end{align*}
proving~\eqref{eq:AL_eps-Omega_eps}.
Finally, using~\eqref{eq:w_n_bounded_L1}, \eqref{eq:w_n_bounded_Linfty} and~\eqref{eq:bongo}, for each $n\in\N$ and $\phi\in C^1_c(\Omega)$  the function
\begin{equation*}
t
\mapsto
\int_\Omega\omega^n(t,\cdot)\,\phi\di x
\in
\mathrm{AC}([0,T];\R)
\end{equation*}
satisfies
\begin{equation}
\label{eq:omega_n_equi-Lip}
\bigg|
\,
\frac{\di}{\di t}\int_\Omega\omega^n(t,\cdot)\,\phi\di x
\,
\bigg|
=
\bigg|
\,
\int_\Omega\omega^n(t,\cdot)\, v^n(t,\cdot)\cdot\nabla\phi \di x
\,
\bigg|
\le 
C\,\|\nabla\phi\|_{L^\infty(\Omega;\,\R^2)}
\end{equation}
for a.e.\ $t\in[0,T]$, where $C>0$ is a constant depending on~$\|\omega_0\|_{L^1(\Omega)}$ and~$\|\omega_0\|_{L^\infty(\Omega)}$ only, proving the validity of~\eqref{eq:AL_equi-cont}.

\smallskip

\textit{Step~3: passage to the limit}.
Thanks to Step~2, we can apply \cref{res:aubin-lions} to the sequence $(\omega^n)_{n\in\N}$ and find a subsequence $(\omega^{n_k})_{k\in\N}$ such that
\begin{equation}
\label{eq:omega_n_weak_converg_omega}
\lim_{k\to+\infty}
\sup_{t\in[0,T]}
\bigg|
\int_\Omega \omega^{n_k}(t,\cdot) \, \phi \di x 
-
\int_\Omega \omega(t,\cdot)\, \phi \di x
\,\bigg|
=0
\end{equation}
for each $\phi\in L^\infty(\Omega)$, for some 
\begin{equation*}
\omega\in L^\infty([0,T];L^1(\Omega))
\cap
C([0,T];L^1(\Omega)-w^\star).
\end{equation*} 
From~\eqref{eq:omega_n_weak_converg_omega}, we see that
\begin{equation*}
\|\omega\|_{L^\infty([0,T];\,L^1(\Omega))}
\le
\sup_{n\in\N}
\|\omega^n\|_{L^\infty([0,T];\,L^1(\Omega))}
\end{equation*}
and
\begin{equation*}
\|\omega\|_{L^\infty([0,T];\,L^\infty(\Omega))}
\le
\sup_{n\in\N}
\|\omega^n\|_{L^\infty([0,T];\,L^\infty(\Omega))},
\end{equation*}
proving~\eqref{eq:existence_bounded_L1} and~\eqref{eq:existence_bounded_Linfty} in virtue of~\eqref{eq:w_n_bounded_L1} and~\eqref{eq:w_n_bounded_Linfty} respectively.
Now we set $\tilde v^n=K\omega^n$ for all $n\in\N$ and
\begin{equation}
\label{eq:salamandra}
v=K\omega\in L^\infty([0,T];C_b(\Omega;\R^2)).
\end{equation}
We observe that, for $\phi\in L^1(\Omega)\cap L^\infty(\Omega)$, we can write
\begin{equation}
\label{eq:K_omega_n-ibp}
\begin{split}
\int_\Omega \phi\,\tilde v^n(t,\cdot)\di x
&=
\int_\Omega \phi\,K\omega^n(t,\cdot)\di x
\\
&=
\int_\Omega \phi(x)\int_\Omega k(x,y)\,\omega^n(t,y)\di y\di x
\\
&=
\int_\Omega \omega^n(t,y)\int_\Omega k(x,y)\,\phi(x)\di x\di y
=
\int_\Omega \omega^n(t,\cdot)\,\tilde K\phi\di y
\end{split}
\end{equation}
for a.e.\ $t\in[0,T]$ and $n\in\N$ by the Fubini Theorem and by~\eqref{eq:Kw_bounded} in \cref{res:holder_bounded},
where we have set 
\begin{equation}
\label{eq:simmetrico}
\tilde K \phi (y) = \int_\Omega k(x,y) \,\varphi(x) \di x,
\quad
x\in\Omega,
\end{equation}
(we do not assume~$k$ to be symmetric in the two variables, but note that the right-hand sides of the estimates~\eqref{eq:k_kernel_bound} and~\eqref{eq:k_kernel_oscillation_bound} in \cref{ass:est} are indeed symmetric). 
In a similar way, we also have
\begin{equation*}
\int_\Omega \phi\,v(t,\cdot)\di x
=
\int_\Omega\omega(t,\cdot)\,\tilde K\phi\di x
\end{equation*}
for a.e.\ $t\in[0,T]$. 
Because of~\eqref{eq:omega_n_weak_converg_omega}, we can thus write
\begin{equation}
\label{eq:ramarro}
\lim_{k\to+\infty}
\int_\Omega \tilde v^{n_k}(t,\cdot) \, \phi \di x 
=
\int_\Omega v(t,\cdot)\, \phi \di x
\end{equation} 
for a.e.\ $t\in[0,T]$, whenever $\phi\in L^1(\Omega)\cap L^\infty(\Omega)$ is given.
In addition, arguing exactly as in Step~1 of the proof of \cref{res:aubin-lions}, given $\phi\in L^\infty(\Omega)$ and $\eps>0$, we can find $\delta>0$ such that 
\begin{equation}
\label{eq:lavagna}
s,t\in[0,T],\
|s-t|<\delta
\implies
\sup_{n\in\N}
\bigg|
\int_\Omega\omega^n(s,\cdot)\,\phi\di x
-
\int_\Omega\omega^n(t,\cdot)\,\phi\di x
\,\bigg|<\eps.
\end{equation} 
Therefore, given $\phi\in L^1(\Omega)\cap L^\infty(\Omega)$ and $\eps>0$, for each $t\in[0,T]$ we can find $\tau_n(t)\in[0,T]$ (defined according to the construction performed in Step~1) such that $\tau_n(t)\le t\le\tau_n(t)+\sfrac1n$ and $v^n(t,\cdot)=v^n(\tau_n(t),\cdot)=K\omega^n(\tau_n(t),\cdot)$, so that   
\begin{align*}
\bigg|
\int_\Omega \phi\, v^n(t,\cdot)\di x
-
\int_\Omega\phi\,\tilde v^n(t,\cdot)\di x
\,\bigg|
&=
\bigg|
\int_\Omega \phi\, v^n(\tau_n(t),\cdot)\di x
-
\int_\Omega\phi\,\tilde v^n(t,\cdot)\di x
\,\bigg|
\\
&=
\bigg|
\int_\Omega\omega^n(\tau_n(t),\cdot)\,\tilde K\phi\di x
-
\int_\Omega\omega^n(t,\cdot)\,\tilde K\phi\di x
\,\bigg|<\eps
\end{align*}
for all $n>\sfrac1\delta$, where $\delta>0$ is given by~\eqref{eq:lavagna} applied to $\tilde K\phi\in L^\infty(\Omega)$. 
Consequently,
because of~\eqref{eq:ramarro}, we get that
\begin{equation}
\label{eq:lucertola}
\lim_{k\to+\infty}
\int_\Omega  v^{n_k}(t,\cdot) \, \phi \di x 
=
\int_\Omega v(t,\cdot)\, \phi \di x
\end{equation}
for a.e.\ $t\in[0,T]$, whenever $\phi\in L^1(\Omega)\cap L^\infty(\Omega)$ is given.   
Now, by Step~1, the sequence $(v^n)_{n\in\N}$ is uniformly equi-bounded and uniformly equi-continuous (uniformly in time) with modulus of continuity~$\ell$.
Thus, by the Arzel\`a--Ascoli Theorem, for a.e.\ $t\in[0,T]$ fixed, we can find a further subsequence $\left(v^{n_{k_j}(t)}\right)_{j\in\N}$ (possibly depending on the chosen time~$t$) and $\tilde v(t,\cdot)\in C_b(\Omega;\R^2)$ such that
\begin{equation}
\label{eq:rettile}
\lim_{k\to+\infty}
\left\|v^{n_{k_j(t)}}(t,\cdot)-\tilde v(t,\cdot)\right\|_{L^\infty_{\loc}(\Omega;\,\R^2)}
=0.
\end{equation}
By combining~\eqref{eq:lucertola} and~\eqref{eq:rettile}, we get that $\tilde v(t,\cdot)=v(t,\cdot)$ and thus 
\begin{equation}
\label{eq:v_n_strong_converg_v}
\lim_{k\to+\infty}
\left\|v^{n_k}(t,\cdot)- v(t,\cdot)\right\|_{L^\infty_{\loc}(\Omega;\,\R^2)}
=0
\end{equation}
for a.e.\ $t\in[0,T]$, that is, the subsequence $(v^{n_k})_{k\in\N}$ is strongly convergent in space independently on the chosen time $t\in[0,T]$.
Hence, we obtain that
\begin{equation}
\label{eq:v_limit_bounded}
\|v(t,\cdot)\|_{L^\infty(\Omega;\,\R^2)}
\lesssim
\|\omega_0\|_{L^1(\Omega)}
+
\|\omega_0\|_{L^\infty(\Omega)}
\end{equation}
and
\begin{equation}
\label{eq:v_limit_log-Lip}
|v(t,x)-v(t,y)|
\lesssim
(\|\omega_0\|_{L^1(\Omega)}
+
\|\omega_0\|_{L^\infty(\Omega)})
\,
\ell(\d(x,y)),
\quad
\forall x,y\in\Omega,
\end{equation}
for a.e.\ $t\in[0,T]$, proving~\eqref{eq:existence_bounded_v_bounded} and~\eqref{eq:existence_bounded_v_log-Lip} respectively. 
Combining~\eqref{eq:omega_n_weak_converg_omega} with~\eqref{eq:v_n_strong_converg_v}, we get that
\begin{equation*}
\lim_{k\to+\infty}
\int_\Omega\omega^{n_k}(t,\cdot) v^{n_k}(t,\cdot)\,\phi\di x
=
\int_\Omega\omega(t,\cdot) v(t,\cdot) \,\phi\di x
\end{equation*} 
for a.e.~$t\in[0,T]$ and all $\phi\in C_c(\Omega)$.
Consequently, passing to the limit as $k\to+\infty$ along the subsequence $(\omega^{n_k},v^{n_k})_{k\in\N}$ in the distributional formulation of~\eqref{eq:bongo}, we conclude that $(\omega,v)$ solves
\begin{equation*}
\begin{cases}
\de_t\omega+\div(v\omega)=0
&
\text{in}\
(0,T)\times\Omega,
\\[2mm]
\omega|_{t=0}=\omega_0
&
\text{on}\
\Omega,	
\end{cases}
\end{equation*}
in the distributional sense, with $v=K\omega$ according to the definition made in~\eqref{eq:salamandra}.

\smallskip

\emph{Step~4: $(\omega,v)$ is Lagrangian}.
We finally prove that the solution $(\omega,v)$ is Lagrangian, i.e.\ $\omega(t,\cdot)=X(t,\cdot)_\#\omega_0$, where $X$ is the flow associated to~$v$.
Note that $X$ is well defined and unique by the classical theory of ODEs, thanks to~\eqref{eq:v_limit_bounded} and~\eqref{eq:v_limit_log-Lip}.

Since $(v^{n_k})_{k\in\N}$ is uniformly equi-bounded and uniformly equi-continuous (uniformly in time), and since the modulus of continuity $\ell$ satisfies the Osgood condition,  the corresponding sequence of flows $(X^{n_k})_{k\in\N}$ is locally uniformly equi-bounded and equi-continuous (uniformly in time) as well. 
Thus, again by the Arzel\`a--Ascoli Theorem (possibly passing to a further subsequence, which we do not relabel), we have that 
\begin{equation*}
\lim_{k\to+\infty}
\|X^{n_k}-X\|_{L^\infty([0,T];\,L^\infty_{\loc}(\Omega;\,\Omega))} = 0
\end{equation*} 
for some $X\in L^\infty([0,T];L^\infty_{\loc}(\Omega;\Omega))$. 
Passing to the limit as $k\to+\infty$ in the expression
\begin{equation*}
X^{n_k}(t,x)=x+\int_0^t v^{n_k}(s,X^{n_k}(s,x))\di s,
\end{equation*}
we get that
\begin{equation*}
X(t,x)=x+\int_0^t v(s,X(s,x))\di s
\end{equation*}
for $x\in\Omega$ and $t\in[0,T]$, so that $X$ must be the (unique) flow associated to $v$.
Therefore
\begin{align*}
\lim_{k\to+\infty}
\int_\Omega \omega^{n_k}(t,\cdot) \, \phi \di x 
=
\lim_{k\to+\infty}
\int_\Omega \omega_0 \, \phi(X^{n_k}(t,\cdot)) \di x 
=
\int_\Omega \omega_0 \, \phi(X(t,\cdot)) \di x 
\end{align*}
for a.e.\ $t\in[0,T]$ and all $\phi\in L^\infty(\Omega)$ by the Dominated Convergence Theorem, and the claimed representation of $\omega$ follows from~\eqref{eq:omega_n_weak_converg_omega}.
The proof is complete. 
\end{proof}

We are now ready to prove the first part of \cref{res:main_existence}, which we recall in the next statement.

\begin{theorem}[Existence in $L^1\cap L^p_{\ul}$ for $p>2$]
\label{res:existence_p_ul}
Let Assumptions~\ref{ass:est} and~\ref{ass:diverg_non-tang} be in force and let $p\in(2,+\infty)$.  
Then there exists a weak solution $(\omega,v)$ of~\eqref{eq:Euler} with vorticity in $L^1\cap L^p_{\ul}$ starting from the initial datum $\omega_0\in L^1(\Omega)\cap L^p_{\ul}(\Omega)$ such that
\begin{equation}
\label{eq:existence_p_ul_L1}
\|\omega\|_{L^\infty([0,T];\,L^1(\Omega))}
\le
\|\omega_0\|_{L^1(\Omega)},
\end{equation}
\begin{equation}
\label{eq:existence_p_ul_Lp_ul}
\|\omega\|_{L^\infty([0,T];\,L^p_{\ul}(\Omega))}
\le
C,
\end{equation} 
\begin{equation}
\label{eq:existence_p_ul_v_bounded}
\|v\|_{L^\infty([0,T];\,L^\infty(\Omega;\,\R^2))}
\le
C,
\end{equation}
and
\begin{equation}
\label{eq:existence_p_ul_v_holder}
|v(t,x)-v(t,y)|
\le
\max\set*{1,\tfrac1{p-2}}
\,
Cp
\,
\d(x,y)^{1-\sfrac 2p},
\quad
\text{for all $x,y\in\Omega$ and a.e.~$t\in [0,T]$,}
\end{equation}
for all $T\in(0,+\infty)$, where $C>0$ only depends on $T$, $p$, $\|\omega_0\|_{L^1(\Omega)}$ and~$\|\omega_0\|_{L^p_{\ul}(\Omega)}$.
\end{theorem}

\begin{proof}
Let $T\in(0,+\infty)$ and $\omega_0\in L^1(\Omega)\cap L^p_{\ul}(\Omega)$ be fixed and define $v_0=K\omega_0$.

\smallskip

\textit{Step~1: construction of $(\omega^n,v^n)_{n\in\N}$}.
For each $n\in\N$, we let $\omega^n_0\in L^1(\Omega)\cap L^\infty(\Omega)$ be the truncation $\omega_0^n=\max\set*{-n,\min\set*{n,\omega_0}}$.
We note that 
\begin{equation*}
\|\omega_0^n\|_{L^1(\Omega)}
\le
\|\omega_0\|_{L^1(\Omega)},
\quad
\text{for all}\ n\in\N,
\end{equation*}
and that
\begin{equation*}
\lim_{n\to+\infty}\|\omega_0^n-\omega_0\|_{L^1(\Omega)}=0.
\end{equation*}
Moreover, we also have that 
\begin{equation}
\label{eq:trunc_omega_0_control_p_ul}
\|\omega_0^n\|_{L^p_{\ul}(\Omega)} \leq \|\omega_0\|_{L^p_{\ul}(\Omega)},
\quad
\text{for all}\ n\in\N.
\end{equation}
For each $n\in\N$, we let $(\omega^n,v^n)_{n\in\N}$ be the Lagrangian weak solution of~\eqref{eq:Euler} in $L^1\cap L^\infty$ with initial datum $\omega^n_0$ given by \cref{res:existence_bounded}.
In particular, we have that 
\begin{equation}
\label{eq:trunc_w_n_L1}
\sup_{n\in\N}
\|\omega^n\|_{L^\infty([0,T];\,L^1(\Omega))}
\le
\|\omega_0\|_{L^1(\Omega)}.
\end{equation}

\smallskip

\textit{Step~2: uniform estimates for $(\omega^n,v^n)_{n\in\N}$}.
Now let $n\in\N$ be fixed.
Since $v^n(t,\cdot)=K\omega^n(t,\cdot)$ for a.e.\ $t\in[0,T]$, by~\eqref{eq:Kw_bounded} in \cref{res:holder_bounded} and by~\eqref{eq:trunc_w_n_L1} in Step~1 we have that
\begin{equation}
\label{eq:trunc_v_n_p}
\begin{split}
\|v^n(t,\cdot)\|_{L^\infty(\Omega;\,\R^2)}
&\lesssim
\max\set*{1,\tfrac1{p-2}}\,
\left(
\|\omega^n(t,\cdot)\|_{L^1(\Omega)}
+
\|\omega^n(t,\cdot)\|_{L^p_{\ul}(\Omega)}
\right)
\\	
&\lesssim
\max\set*{1,\tfrac1{p-2}}\,
\left(
\|\omega_0\|_{L^1(\Omega)}
+
\|\omega^n(t,\cdot)\|_{L^p_{\ul}(\Omega)}
\right)
\\	
&\lesssim
\max\set*{1,\tfrac1{p-2}}
\,
\max\set*{1,\|\omega_0\|_{L^1(\Omega)}}
\,
\left(
1
+
\|\omega^n(t,\cdot)\|_{L^p_{\ul}(\Omega)}
\right)
\end{split}
\end{equation}
for a.e.\ $t\in[0,T]$.
We now consider the function
\begin{equation}
\label{eq:trunc_v_n_def_Rn}
R_n(t)
=
\int_0^t\|v^n(s,\cdot)\|_{L^\infty(\Omega;\,\R^2)}\di s
\end{equation}
defined for $t\in[0,T]$. 
Let $X^n$ be the flow associated to the velocity $v^n$.
Since 
\begin{equation*}
\d(X_t^n(x),x)
\le
R_n(t)
\quad
\text{for all}\ x\in\Omega, 
\end{equation*}
by exploiting~\eqref{eq:trunc_w_n_L1} in Step~1 and \eqref{eq:null_div_K} we can estimate
\begin{equation}
\label{eq:trunc_omega_n_p}
\begin{split}
\|\omega^n(t,\cdot)\|_{L^p_{\ul}(\Omega)}
&\le
\exp\left(\tfrac T{p'}\|\div v^n\|_{L^\infty([0,T];\,L^\infty(\Omega))}\right)
\|\omega_0\|_{L^p_{\ul,1+R_n(t)}(\Omega)}
\\
&\le
\exp\left(\tfrac T{p'}\,C_3\|\omega^n\|_{L^\infty([0,T];\,L^1(\Omega))}\right)
\|\omega_0\|_{L^p_{\ul,1+R_n(t)}(\Omega)}
\\
&\le
\exp\left(\tfrac T{p'}\,C_3\|\omega_0\|_{L^1(\Omega)}\right)
\|\omega_0\|_{L^p_{\ul,1+R_n(t)}(\Omega)}
\end{split}
\end{equation}
for all $t\in[0,T]$, where $p'=\sfrac{p}{(p-1)}\in(1,2)$.
By~\eqref{eq:Lp_ul_change_radius} we have that
\begin{equation*}
\|\omega_0\|_{L^p_{\ul,1+R_n(t)}(\Omega)}
\lesssim
(1+R_n(t))^{\sfrac 2p}
\,
\|\omega_0\|_{L^p_{\ul}(\Omega)},
\end{equation*}
and thus
\begin{equation}
\label{eq:trunc_omega_n_change_radius}
\|\omega^n(t,\cdot)\|_{L^p_{\ul}(\Omega)}
\le
\exp\left(\tfrac T{p'}\,C_3\|\omega_0\|_{L^1(\Omega)}\right)
\,
(1+R_n(t))^{\sfrac 2p}
\,
\|\omega_0\|_{L^p_{\ul}(\Omega)}
\end{equation}
for all $t\in[0,T]$. 
Therefore, by combining~\eqref{eq:trunc_v_n_p}, \eqref{eq:trunc_v_n_def_Rn}, \eqref{eq:trunc_omega_n_p} and~\eqref{eq:trunc_omega_n_change_radius}, we get
\begin{equation}
\label{eq:key_est_Rn}
R_n'(t)
\lesssim
C
\left(
1
+
\|\omega_0\|_{L^p_{\ul}(\Omega)}
\,
(1+R_n(t))^{\sfrac 2p}
\right)
\end{equation}
for a.e.\ $t\in(0,T)$, where
\begin{equation*}
C
=
T\max\set*{1,\tfrac1{p-2}}
\,
\max\set*{1,\|\omega_0\|_{L^1(\Omega)}}
\,
\exp\left(\tfrac T{p'}\,C_3\|\omega_0\|_{L^1(\Omega)}\right).	
\end{equation*}
From inequality~\eqref{eq:key_est_Rn} we thus get that
\begin{equation}
\label{eq:magic_constant_p}
R_n(t)
\le
C(p,T,\|\omega_0\|_{L^1(\Omega)},\|\omega_0\|_{L^p_{\ul}(\Omega)})
\end{equation}
for all $t\in[0,T]$, where the constant appearing in the right-hand side does not depend on the choice of~$n\in\N$.
Consequently, by~\eqref{eq:trunc_omega_n_change_radius} we get that
\begin{equation}
\label{eq:trunc_w_n_Lp_ul}
\sup_{n\in\N}
\|\omega^n\|_{L^\infty([0,T];\,L^p_{\ul}(\Omega))}
\le
C(p,T,\|\omega_0\|_{L^1(\Omega)},\|\omega_0\|_{L^p_{\ul}(\Omega)})
\end{equation}
and then, using \eqref{eq:trunc_v_n_p}, we deduce
\begin{equation}
\label{eq:trunc_v_n_unif_est_ul_p}
\sup_{n\in\N}
\|v^n\|_{L^\infty([0,T];\,L^\infty(\Omega;\,\R^2))}
\le
C(p,T,\|\omega_0\|_{L^1(\Omega)},\|\omega_0\|_{L^p_{\ul}(\Omega)}).
\end{equation}

\smallskip

\textit{Step~3: properties of $(\omega^n)_{n\in\N}$}.
We now claim that the sequence $(\omega^n)_{n\in\N}$ satisfies the hypotheses of \cref{res:aubin-lions}.
Indeed, property~\eqref{eq:AL_bounded} follows from~\eqref{eq:trunc_w_n_L1} in Step~1. 
Property~\eqref{eq:AL_eps-Omega_eps} can be proved as in Step~2 of the proof of \cref{res:existence_bounded}, thanks to the uniform bound~\eqref{eq:trunc_v_n_unif_est_ul_p} proved in Step~2.
In particular, for each $\eps>0$ we can find $R>0$ such that
\begin{equation}
\label{eq:trunc_w_n_ball_eps}
\sup_{n\in\N}
\sup_{t\in[0,T]}
\int_{\Omega\setminus B_R}|\omega^n(t,\cdot)|\di x
<
\eps.
\end{equation}
Also property~\eqref{eq:AL_equi-cont} can be proved as in Step~2 of the proof of \cref{res:existence_bounded}, again thanks to the uniform bound in~\eqref{eq:trunc_w_n_L1} and \eqref{eq:trunc_v_n_p} and since $(\omega^n,v^n)$ solves~\eqref{eq:Euler} for each~$n\in\N$.
We are thus left to show property~\eqref{eq:AL_eps-delta}.
To this aim, let $\eps>0$ and $A\subset\Omega$.
Letting $R>0$ be the radius given by~\eqref{eq:trunc_w_n_ball_eps}, we can write
\begin{equation}
\label{eq:trunc_w_n_eps-delta_1}
\begin{split}
\int_A|\omega^n(t,\cdot)|\di x
&=
\int_{A\cap B_R}|\omega^n(t,\cdot)|\di x
+
\int_{A\setminus B_R}|\omega^n(t,\cdot)|\di x
\le
\int_{A\cap B_R}|\omega^n(t,\cdot)|\di x
+
\eps
\end{split}
\end{equation}
for all $t\in[0,T]$ and $n\in\N$. 
Moreover, thanks to the uniform bound~\eqref{eq:trunc_w_n_Lp_ul} and the inequality~\eqref{eq:Lp_ul_change_radius}, we can estimate
\begin{equation}
\label{eq:trunc_w_n_eps-delta_2}
\begin{split}
\sup_{n\in\N}
\sup_{t\in[0,T]}
\int_{A\cap B_R}|\omega^n(t,\cdot)|\di x
&\le
|A|^{\sfrac{1}{p'}}
\sup_{n\in\N}
\|\omega^n\|_{L^\infty([0,T];\,L^p(B_R))}
\\
&\le
|A|^{\sfrac{1}{p'}}
\sup_{n\in\N}
\|\omega^n\|_{L^\infty([0,T];\,L^p_{\ul,R}(\Omega))}
\\
&\lesssim
R^{\sfrac{2}{p}}\,
|A|^{\sfrac{1}{p'}}
\sup_{n\in\N}
\|\omega^n\|_{L^\infty([0,T];\,L^p_{\ul}(\Omega))}
\\
&\le
R^{\sfrac{2}{p}}\,
C(p,T,\|\omega_0\|_{L^1(\Omega)},\|\omega_0\|_{L^p_{\ul,1}(\Omega)})
\,
|A|^{\sfrac{1}{p'}},
\end{split}
\end{equation}
where the implicit (geometric) constant in the intermediate inequality does not depend on~$\eps$, and as usual $p'=\sfrac{p}{(p-1)}$. 
Property~\eqref{eq:AL_eps-delta} thus follows by combining~\eqref{eq:trunc_w_n_eps-delta_1} and~\eqref{eq:trunc_w_n_eps-delta_2}.

\smallskip

\textit{Step~4: construction of $(\omega,v)$}.
Thanks to Step~3, we can apply \cref{res:aubin-lions} to the sequence $(\omega^n)_{n\in\N}$ and find a subsequence $(\omega^{n_k})_{k\in\N}$ such that
\begin{equation}
\label{eq:trunc_omega_n_weak_converg_omega}
\lim_{k\to+\infty}
\sup_{t\in[0,T]}
\bigg|
\int_\Omega \omega^{n_k}(t,\cdot) \, \phi \di x 
-
\int_\Omega \omega(t,\cdot)\, \phi \di x
\,\bigg|
=0
\end{equation}
for all $\phi\in L^\infty(\Omega)$, for some
\begin{equation*}
\omega\in L^\infty([0,T];L^1(\Omega))
\cap
C([0,T];L^1(\Omega)-w^\star).	
\end{equation*} 
From~\eqref{eq:trunc_omega_n_weak_converg_omega} it follows that
\begin{equation}
\label{eq:limit_omega_trunc_1}
\|\omega\|_{L^\infty([0,T];\,L^1(\Omega))}
\le
\sup_{n\in\N}
\|\omega^n\|_{L^\infty([0,T];\,L^1(\Omega))}
\end{equation}
and
\begin{equation}
\label{eq:limit_omega_trunc_p_ul}
\|\omega\|_{L^\infty([0,T];\,L^p_{\ul}(\Omega))}
\le
\sup_{n\in\N}
\|\omega^n\|_{L^\infty([0,T];\,L^p_{\ul}(\Omega))},
\end{equation}
proving~\eqref{eq:existence_p_ul_L1} and~\eqref{eq:existence_p_ul_Lp_ul} in virtue of~\eqref{eq:trunc_w_n_L1} and~\eqref{eq:trunc_w_n_Lp_ul} respectively.
Now, since 
\begin{equation}
\label{eq:trunc_v_n_is_K_trunc_omega_n}
v^n(t,\cdot)=K\omega^n(t,\cdot)
\quad
\text{for a.e.\ $t\in[0,T]$ and all $n\in\N$}, 
\end{equation}
by~\eqref{eq:trunc_omega_n_weak_converg_omega} and the Fubini Theorem we get that
\begin{equation}
\label{eq:trunc_v_n_full_limit}
\lim_{k\to+\infty}
\int_\Omega v^{n_k}(t,\cdot)\,\phi \di x
=
\lim_{k\to+\infty}
\int_\Omega \omega^{n_k}(t,\cdot)\,\tilde K\phi \di x
=
\int_\Omega \omega(t,\cdot)\,\tilde K\phi \di x
\end{equation}
for a.e.\ $t\in[0,T]$ and all $\phi\in C_c(\Omega)$, where $\tilde K$ is as in~\eqref{eq:simmetrico}. 
From Step~2, we already know that the sequence $(v^n)_{n\in\N}$ is uniformly equi-bounded (uniformly in time).
By recalling~\eqref{eq:trunc_v_n_is_K_trunc_omega_n} and by combining~\eqref{eq:trunc_w_n_L1} and~\eqref{eq:trunc_w_n_Lp_ul} with~\eqref{eq:Kw_holder} of \cref{res:holder_bounded}, we get that the sequence $(v^n)_{n\in\N}$ is also uniformly equi-H\"older-continuous (uniformly in time).
Therefore, by the Arzel\`a--Ascoli Theorem, for a.e.\ $t\in[0,T]$ we can find a subsequence $(n_{k_j(t)})_{j\in\N}$ (which a priori may depend on the chosen~$t$) and a function 
$v(t,\cdot) \in L^\infty(\Omega;\R^2)$ such that 
\begin{equation}
\label{eq:trunc_v_n_subseq_t_strong_conv}
\lim_{j\to+\infty}
\big\|v^{n_{k_j(t)}}(t,\cdot)-v(t,\cdot)\big\|_{L^\infty_{\loc}(\Omega;\,\R^2)}
=0.
\end{equation}
Consequently, for a.e.~$t\in[0,T]$ we must have that
\begin{equation*}
\lim_{j\to+\infty}
\int_\Omega v^{n_{k_j(t)}}(t,\cdot)\,\phi\di x
=
\int_\Omega v(t,\cdot)\,\phi\di x
\end{equation*}
for all $\phi\in C_c(\Omega)$.
Thanks to~\eqref{eq:trunc_v_n_full_limit}, we thus have $v(t,\cdot)=K\omega(t,\cdot)$ for a.e.\ $t\in[0,T]$ and hence, in virtue of \cref{res:holder_bounded} again and the bounds~\eqref{eq:limit_omega_trunc_1} and~\eqref{eq:limit_omega_trunc_p_ul}, we immediately get that
\begin{equation}
\label{eq:v_limit_p_ul_bounded}
\|v\|_{L^\infty([0,T];\,L^\infty(\Omega;\,\R^2))}
\lesssim
C
\end{equation}
and
\begin{equation}
\label{eq:v_limit_p_ul_holder}
|v(t,x)-v(t,y)|
\lesssim
\max\set*{1,\tfrac1{p-2}}
\,
Cp
\,
\d(x,y)^{1-\sfrac 2p},
\quad
\text{for all $x,y\in\Omega$ and a.e.~$t\in [0,T]$,}
\end{equation}
where $C=C(p,T,\|\omega_0\|_{L^1(\Omega)},\|\omega_0\|_{L^p_{\ul}(\Omega)})$ is the constant appearing in~\eqref{eq:trunc_v_n_unif_est_ul_p}, 
proving~\eqref{eq:existence_p_ul_v_bounded} and~\eqref{eq:existence_p_ul_v_holder} respectively. 
In addition, by combining~\eqref{eq:trunc_v_n_full_limit} with~\eqref{eq:trunc_v_n_subseq_t_strong_conv}, we easily see that, in fact,
\begin{equation}
\label{eq:trunc_v_n_strong_conv}
\lim_{k\to+\infty}
\big\|v^{n_k}(t,\cdot)-v(t,\cdot)\big\|_{L^\infty_{\loc}(\Omega;\,\R^2)}
=0
\end{equation}
for a.e.\ $t\in[0,T]$, that is, the subsequence can be chosen independently of time.
Consequently, given any $\phi\in C_c(\Omega)$, from~\eqref{eq:trunc_w_n_L1}, \eqref{eq:trunc_omega_n_weak_converg_omega}, and~\eqref{eq:trunc_v_n_strong_conv}, we immediately get 
\begin{equation*}
\lim_{k\to+\infty}
\int_\Omega \omega^{n_k}(t,\cdot) v^{n_k}(t,\cdot)\,\phi\di x
=
\int_\Omega \omega(t,\cdot) v(t,\cdot)\,\phi\di x
\end{equation*}   
for a.e.\ $t\in[0,T]$.
Therefore, passing to the limit as $k\to+\infty$ along the subsequence $(\omega^{n_k},v^{n_k})_{k\in\N}$ in the weak formulation of~\eqref{eq:Euler}, we conclude that $(\omega,v)$ solves~\eqref{eq:Euler} in the distributional sense and the proof is complete.
\end{proof}

\begin{remark}
Inequality~\eqref{eq:trunc_omega_n_change_radius} and the overall strategy developed in Step~2 of the above proof can be seen as a Lagrangian reformulation  of the Eulerian a priori estimates established in~\cite{T04}*{Lemma~1.4} and in~\cite{CMZ19}*{Proposition~3.1}. 
\end{remark}

We can now conclude this section by proving the second part of \cref{res:main_existence}, which we recall in the next statement. 

\begin{theorem}[Existence in $L^1\cap Y^\Theta_{\ul}$ for any $\Theta$]
\label{res:existence_Theta}
Let Assumptions~\ref{ass:est} and~\ref{ass:diverg_non-tang} be in force.
Then there exists a weak solution $(\omega,v)$ of~\eqref{eq:Euler} in $L^1\cap Y^\Theta_{\ul}$ with initial datum $\omega_0\in L^1(\Omega)\cap Y^\Theta_{\ul}(\Omega)$ such that
\begin{equation}
\label{eq:existence_Theta_ul_L1}
\|\omega\|_{L^\infty([0,T];\,L^1(\Omega))}
\le
\|\omega_0\|_{L^1(\Omega)},
\end{equation}
\begin{equation}
\label{eq:existence_Theta_ul_Theta_ul}
\|\omega\|_{L^\infty([0,T];\,Y^\Theta_{\ul}(\Omega))}
\le
C,
\end{equation} 
\begin{equation}
\label{eq:existence_Theta_ul_v_bounded}
\|v\|_{L^\infty([0,T];\,L^\infty(\Omega;\,\R^2))}
\le
C,
\end{equation}
and
\begin{equation}
\label{eq:existence_Theta_ul_v_almost-lip}
|v(t,x)-v(t,y)|
\le
C
\,
\phi_\Theta(\d(x,y)),
\quad
\text{for all $x,y\in\Omega$ and a.e.~$t\in [0,T]$,}
\end{equation}
for all $T\in(0,+\infty)$, where $C>0$ only depends on $T$, $\|\omega_0\|_{L^1(\Omega)}$ and~$\|\omega_0\|_{Y^\Theta_{\ul}(\Omega)}$.
Moreover, if the growth function~$\Theta$ satisfies~\eqref{eq:Yudovich_condition_Theta}, then $(\omega,v)$ is Lagrangian.
\end{theorem}

\begin{proof}
Since $\omega_0\in L^1(\Omega)\cap L^p_{\ul}(\Omega)$ for all $2<p<\infty$, we can apply \cref{res:existence_p_ul}, the only thing we have to check being the behavior of the constant $C$ appearing in~\eqref{eq:existence_p_ul_Lp_ul}, \eqref{eq:existence_p_ul_v_bounded} and~\eqref{eq:existence_p_ul_v_holder} for large values of~$p$.
A quick inspection of the proof of \cref{res:existence_p_ul} immediately shows that it is enough to control the function
\begin{equation*}
p\mapsto C(p,T,\|\omega_0\|_{L^1(\Omega)},\|\omega_0\|_{L^p_{\ul}(\Omega)})
\end{equation*}
appearing in the right-side of~\eqref{eq:magic_constant_p} in Step~2 of the proof of \cref{res:existence_p_ul}.
However, with the same notation of the proof of \cref{res:existence_p_ul}, we can replace~\eqref{eq:trunc_omega_0_control_p_ul} with
\begin{equation*}
\|\omega_0^n\|_{Y^\Theta_{\ul}(\Omega)} \leq \|\omega_0\|_{Y^\Theta_{\ul}(\Omega)}
\quad
\text{for all}\ n\in\N.
\end{equation*}
As a consequence, we can repeat the argument of Step~2 of the proof of \cref{res:existence_p_ul} line by line and replace~\eqref{eq:key_est_Rn} with
\begin{equation*}
R_n'(t)
\lesssim
C
\left(
1
+
\|\omega_0\|_{Y^\Theta_{\ul}(\Omega)}
\,
(1+R_n(t))
\right)
\end{equation*}
for all $t\in(0,T)$, where now
\begin{equation*}
C
=
\max\set*{1,\|\omega_0\|_{L^1(\Omega)}}
\,
\exp\left(TC_3\|\omega_0\|_{L^1(\Omega)}\right),
\end{equation*}
and the first part of the statement readily follows.

If~$\Theta$ satisfies~\eqref{eq:Yudovich_condition_Theta}, then in Step~4 of the proof of \cref{res:existence_p_ul} the sequence $(v^n)_{n\in\N}$ is also uniformly equi-$\phi_\Theta$-continuous (uniformly in time), i.e.\ 
\begin{equation*}
|v^n(t,x)-v^n(t,y)|
\le
C
\,
\phi_\Theta(\d(x,y)),
\quad
\text{for all $x,y\in\Omega$ and a.e.~$t\in [0,T]$,}
\end{equation*} 
for all $n\in\N$, where $C>0$ is as above. 
Since $\phi_\Theta$ satisfies the Osgood condition~\eqref{eq:Osgood}, the sequence $(X^n)_{n\in\N}$ of the (unique) associated flows is locally uniformly equi-bounded and equi-continuous (uniformly in time) and we can argue as in Step~4 of the proof of \cref{res:existence_bounded}.
The proof is complete.
\end{proof}

\section{Uniqueness of weak solutions (\texorpdfstring{\cref{res:main_uniqueness}}{Theorem 1.6})}
\label{sec:uniqueness}

In this section, we prove the uniqueness of Lagrangian weak solutions of the Euler equations~\eqref{eq:Euler} in $L^1\cap Y^\Theta_{\ul}$ under the Osgood condition~\eqref{eq:Osgood} and the concavity property of the modulus of continuity~$\phi_\Theta$ defined in~\eqref{eq:def_phi_Theta}, establishing \cref{res:main_uniqueness}. We recall the result in the next statement.

\begin{theorem}[Lagrangian uniqueness in $L^1\cap Y^\Theta_{\ul}$]
\label{res:uniqueness}
Let Assumptions~\ref{ass:est} and~\ref{ass:diverg_non-tang} be in force.
If the growth function $\Theta$ satisfies~\eqref{eq:Yudovich_condition_Theta} and the function $\phi_\Theta$ defined in~\eqref{eq:def_phi_Theta} is concave on $[0,+\infty)$, then there exists at most one Lagrangian weak solution $(\omega,v)$ of~\eqref{eq:Euler} with vorticity in $L^1\cap Y^\Theta_{\ul}$ starting from a given initial datum $\omega_0\in L^1(\Omega)\cap Y^\Theta_{\ul}(\Omega)$. 
\end{theorem}

\begin{proof}
Let $(\omega,v)$ and $(\tilde \omega,\tilde v)$ be two Lagrangian weak solutions of~\eqref{eq:Euler} with vorticity in $L^1\cap Y^\Theta_{\ul}$ starting from the same initial datum $\omega_0\in L^1(\Omega)\cap Y^\Theta_{\ul}(\Omega)$ and let $T\in(0,+\infty)$ be fixed.
We write $\omega(t,\cdot)=X(t,\cdot)_{\#}\omega_0$ and $\tilde\omega(t,\cdot)=\tilde X(t,\cdot)_{\#}\omega_0$ for $t\in[0,T]$, where $X$ and $\tilde X$ are the (unique) flows associated to $v$ and $\tilde v$ respectively.
Let $\eta\in L^1(\Omega)\cap L^\infty(\Omega)$ such that $\eta(x)>0$ for all $x\in\Omega$,
let $\bar{\omega}=|\omega_0|+\eta$, note that $\bar{\omega}\in L^1(\Omega)\cap Y^\Theta_{\ul}(\Omega)$, and define the finite measure $\mu=\bar{\omega}\,\mathscr{L}^2\in\mathcal M(\Omega)$.
Now, for $x\in\Omega$, we can estimate
\begin{align*}
\d(X(t,x),\tilde X(t,x))
&\le
\int_0^t|v(s,X(s,x))-\tilde v(s,\tilde X(s,x))|\di s
\\
&\le 
\int_0^t|v(s,X(s,x))-v(s,\tilde X(s,x))|\di s
+
\int_0^t|v(s,\tilde X(s,x))-\tilde v(s,\tilde X(s,x))|\di s
\end{align*}
for all $t\in[0,T]$.
On the one side, by~\eqref{eq:Kw_almost_Lip} in \cref{res:almost_Lip} and by the fact that $(\omega,v)$ is a Lagrangian weak solution of~\eqref{eq:Euler} with vorticity in $L^1\cap Y^\Theta_{\ul}$, we have
\begin{equation*}
|v(s,X(s,x))-v(s,\tilde X(s,x))|
\lesssim
A \, \phi_\Theta(\d(X(s,x),\tilde X(s,x)))
\end{equation*} 
for a.e.\ $s\in[0,T]$, where $A>0$ only depends on $T$, $\|\omega\|_{L^\infty([0,T];L^1(\Omega))}$, and~$\|\omega\|_{L^\infty([0,T];Y^\Theta_{\ul}(\Omega))}$.
On the other side, we have
\begin{align*}
|v(s,\tilde X(s,x))&-\tilde v(s,\tilde X(s,x))|
=
|(K\omega)(s,\tilde X(s,x))-(K\tilde\omega)(s,\tilde X(s,x))|
\\
&=
\left|
\int_\Omega k(\tilde X(s,x),y)\,\omega(s,y)\di y
-
\int_\Omega k(\tilde X(s,x),y)\,\tilde\omega(s,y)\di y
\right|
\\
&=
\left|
\int_\Omega k(\tilde X(s,x),X(s,y))\,\omega_0(y)\di y
-
\int_\Omega k(\tilde X(s,x),\tilde X(s,y))\,\omega_0(y)\di y
\right|
\\
&\le
\int_\Omega |k(\tilde X(s,x),X(s,y))-k(\tilde X(s,x),\tilde X(s,y))|\,|\omega_0(y)|\di y
\end{align*}
for a.e.\ $s\in[0,T]$.
Therefore, we get that
\begin{align*}
\int_\Omega\d(X(t,x) &, \tilde X(t,x))
\di\mu(x) 
\le
\int_0^t\int_\Omega
|v(s,X(s,x))-v(s,\tilde X(s,x))|\di\mu(x) \di s
\\
&\quad+
\int_0^t\int_\Omega|v(s,\tilde X(s,x))-\tilde v(s,\tilde X(s,x))|\di\mu(x) \di s
\\
&\le
A 
\int_0^t\int_\Omega
\phi_\Theta(\d(X(s,x),\tilde X(s,x)))
\di\mu(x) \di s
\\
&\quad+
\int_0^t\int_\Omega
\int_\Omega |k(\tilde X(s,x),X(s,y))-k(\tilde X(s,x),\tilde X(s,y))|\,|\omega_0(y)|\di y
\di\mu(x) \di s
\\
&=
A 
\int_0^t\int_\Omega
\phi_\Theta(\d(X(s,x),\tilde X(s,x)))
\di\mu(x)\di s
\\
&\quad+
\int_0^t\int_\Omega|\omega_0(y)|
\int_\Omega |k(\tilde X_s(x),X_s(y))-k(\tilde X_s(x),\tilde X_s(y))|\di\mu(x)\di y
\di s
\end{align*}
for all $t\in[0,T]$.
Thanks to \eqref{eq:Kw_almost_Lip_strong} in \cref{rem:K_stronger} (applied to the operator~$\tilde K$ defined in~\eqref{eq:simmetrico}), we can thus estimate
\begin{align*}
\int_\Omega |k(\tilde X(s,x)&,X(s,y))-k(\tilde X(s,x) ,\tilde X(s,y))|\di\mu(x)
\\
&=
\int_\Omega |k(x,X(s,y))-k(x,\tilde X(s,y))|\di\,(\tilde X(s,\cdot))_\#\mu(x)
\\
&=
\int_\Omega |k(x,X(s,y))-k(x,\tilde X(s,y))|\,|\bar\omega(s,x)|\di x \\
&\lesssim
\left(
\|\bar\omega(s,\cdot)\|_{L^1(\Omega)}
+
\|\bar\omega(s,\cdot)\|_{Y^\Theta_{\ul}(\Omega)}
\right)
\phi_\Theta(\d(X(s,y),\tilde X(s,y)))
\end{align*}
for a.e.~$s\in[0,T]$ and $y\in\Omega$, where $\bar\omega(s,\cdot)=\tilde X(s,\cdot)_\#\mu$. 
Now, since $\bar\omega=|\omega_0|+\eta$, we can write 
$\bar\omega(s,\cdot)
=
|\tilde\omega(s,\cdot)|
+
\tilde\eta(s,\cdot)
$
for all $s\in[0,T]$, where $\tilde\eta(s,\cdot)=\tilde X(s,\cdot)_\#\eta$. 
Consequently, recalling that $\eta\in L^1(\Omega)\cap L^\infty(\Omega)$ by definition, we can estimate
\begin{align*}
\|\tilde\eta(t,\cdot)\|_{L^\infty(\Omega)}
\le
\exp
\left(
\int_0^t
\|\div\tilde v(s,\cdot)\|_{L^\infty(\Omega)}
\di s
\right)
\|\eta\|_{L^\infty(\Omega)}
\le
\exp
\left(
C_3\|\tilde\omega\|_{L^\infty([0,T];\,L^1(\Omega))}
\right)
\|\eta\|_{L^\infty(\Omega)}
\end{align*}
for all $t\in[0,T]$ according to~\eqref{eq:null_div_K} in \cref{ass:diverg_non-tang}, and thus
\begin{equation*}
\|\bar\omega(s,\cdot)\|_{L^1(\Omega)}
+
\|\bar\omega(s,\cdot)\|_{Y^\Theta_{\ul}(\Omega)}
\le 
B
\end{equation*}
for all $s\in[0,T]$, where $B>0$ only depends on $T$, $\|\eta\|_{L^1(\Omega)}$, $\|\eta\|_{L^\infty(\Omega)}$, $\|\tilde\omega\|_{L^\infty([0,T];\,L^1(\Omega))}$, and~$\|\tilde\omega\|_{L^\infty([0,T];\,Y^\Theta_{\ul}(\Omega))}$.
Therefore, recalling that $|\omega_0|\le\bar\omega$ by construction, we conclude that
\begin{align*}
\int_\Omega 
& 
\, \d(X(t,x),\tilde X(t,x))\di\mu(x) \\
& \le
A\int_0^t\int_\Omega
\phi_\Theta(\d(X(s,x),\tilde X(s,x)))
\di\mu(x)\di s
\\
&\quad+
\int_0^t \int_\Omega|\omega_0(y)|\,
\int_\Omega |k(\tilde X(s,x),X(s,y))-k(\tilde X(s,x),\tilde X(s,y))|\di\mu(x)\di y
\di s
\\
&\le
A
\int_0^t\int_\Omega
\phi_\Theta(\d(X(s,x),\tilde X(s,x)))
\di\mu(x)\di s \\
&\quad+
B
\int_0^t\int_\Omega
\phi_\Theta(\d(X(s,y),\tilde X(s,y)))\,
|\omega_0(y)|
\di y
\di s
\\
&\lesssim
C
\int_0^t\int_\Omega
\phi_\Theta(\d(X(s,x),\tilde X(s,x)))
\di\mu(x)\di s
\end{align*}
for all $t\in[0,T]$, where $C>0$ only depends on $T$, $\|\eta\|_{L^1(\Omega)}$, $\|\eta\|_{L^\infty(\Omega)}$,  
$\|\omega\|_{L^\infty([0,T];\,L^1(\Omega))}$, $\|\omega\|_{L^\infty([0,T];\,Y^\Theta_{\ul}(\Omega))}$,
$\|\tilde \omega\|_{L^\infty([0,T];\,L^1(\Omega))}$, and $\|\tilde\omega\|_{L^\infty([0,T];\,Y^\Theta_{\ul}(\Omega))}$.
Since $\mu(\Omega)<+\infty$ and $\phi_\Theta$ is concave, by Young's inequality we thus get that
\begin{align*}
\aint_\Omega\,\d(X(t,\cdot),\tilde X(t,\cdot))\di\mu
&\lesssim
C
\int_0^t\phi_\Theta\left(
\aint_\Omega
\d(X(s,\cdot),\tilde X(s,\cdot))
\di\mu 
\right)
ds
\end{align*}
for all $t\in[0,T]$.
Hence, since $\phi_\Theta$ satisfies the Osgood condition~\eqref{eq:Osgood}, we conclude that
\begin{equation*}
\aint_\Omega\,\d(X(t,\cdot),\tilde X(t,\cdot))\di\mu=0
\quad
\text{for all}\ t\in[0,T],
\end{equation*}
proving that $X(t,x)=\tilde X(t,x)$ for all $t\in[0,T]$ and all $x\in\Omega$.
So we must have that $\omega(t,\cdot)=\tilde\omega(t,\cdot)$ for all $t\in[0,T]$ and, since $T\in(0,+\infty)$ was arbitrary, the conclusion follows.
\end{proof}

\appendix

\section{An Aubin--Lions-like Lemma}
\label{sec:aubin-lions}

In this section, we prove a simple Aubin--Lions-like Lemma.
This result is needed in \cref{sec:existence} for the construction of the weak solutions of the Euler equations~\eqref{eq:Euler}.
The proof exploits a combination of the Dunford--Pettis Theorem and the Arzel\`a--Ascoli Compactness Theorem together with some standard approximation arguments.

We were not able to find the  result below  in this precise form in the literature, so we prove it here from scratch for the reader's convenience. 
We underline that \cref{res:aubin-lions} just assumes \emph{weak} compactness in space, while one usually deals with \emph{strong} compactness in space.
For a result very similar to \cref{res:aubin-lions}, see~\cite{F04}*{Corollary~2.1} (we thank Stefano Spirito for pointing this reference to us).

\begin{theorem}
\label{res:aubin-lions}
Let $\Omega\subset\R^N$ be an open set and $T\in(0,+\infty)$. 
Let 
$(f^n)_{n\in\N}\subset L^\infty([0,T];L^1(\Omega))$
be a bounded sequence which is equi-integrable in space uniformly in time, that is,
\begin{equation}
\label{eq:AL_bounded}
\sup_{n\in\N}\|f^n\|_{L^\infty([0,T];\,L^1(\Omega))}<+\infty,
\end{equation}
\begin{equation}
\label{eq:AL_eps-delta}
\forall\eps>0\ \exists\delta>0\ 
\colon
A\subset\Omega,\ |A|<\delta
\implies
\sup_{n\in\N}\|f^n\|_{L^\infty([0,T];\,L^1(A))}<\eps
\end{equation}
and
\begin{equation}
\label{eq:AL_eps-Omega_eps}
\forall\eps>0\ \exists\Omega_\eps\subset\Omega\
\text{with}\ |\Omega_\eps|<+\infty\ \colon
\sup_{n\in\N}\|f^n\|_{L^\infty([0,T];\,L^1(\Omega\setminus\Omega_\eps))}<\eps.
\end{equation}
Assume that, for each $\phi\in C^\infty_c(\Omega)$, the functions $F_n[\phi]\colon[0,T]\to\R$, given by
\begin{equation}
\label{eq:AL_def_Fn_phi}
F_n[\phi](t)=\int_\Omega f^n(t,\cdot)\,\phi\di x,
\quad
t\in[0,T],
\end{equation}
are uniformly equi-continuous on $[0,T]$, that is,
\begin{equation}
\label{eq:AL_equi-cont}
\forall\eps>0\
\exists\delta>0\
\colon\
s,t\in[0,T],\ |s-t|<\delta
\implies
\sup_{n\in\N}|F_n[\phi](s)-F_n[\phi](t)|<\eps.
\end{equation}
Then there exist a subsequence $(f^{n_k})_{k\in\N}$ and a function 
\begin{equation}
\label{eq:AL_prop_limit_f}
f\in
L^\infty([0,T];L^1(\Omega))
\cap
C([0,T];L^1(\Omega)-w^\star)
\end{equation}
that is,
\begin{equation*}
t \mapsto \int_{\Omega} f(t,\cdot) \, \phi \di x
\in
C([0,T];\R)
\quad
\text{for every}\
\phi\in L^\infty(\Omega),
\end{equation*}
such that
\begin{equation}
\label{eq:AL_thesis}
\lim_{k\to+\infty}
\sup_{t\in[0,T]}
\left|\,
\int_\Omega f^{n_k}(t,\cdot) \, \phi \di x 
-
\int_\Omega f(t,\cdot) \, \phi \di x
\,\right|
=0
\end{equation}
for all $\phi\in L^\infty(\Omega)$.
\end{theorem}

\begin{proof}
We divide the proof in four steps.

\smallskip

\textit{Step~1: equi-continuity testing against  $L^\infty(\Omega)$}.
We claim that~\eqref{eq:AL_equi-cont} actually holds for each $\phi\in L^\infty(\Omega)$.
To prove this statement, we distinguish two cases.

\smallskip

\emph{Case~1}.
Let us prove~\eqref{eq:AL_equi-cont} for any $\phi\in C_c(\Omega)$ at first.
Let $\eps>0$ be fixed.
We can find $\psi\in C^\infty_c(\Omega)$ such that 
$\|\psi-\phi\|_{L^\infty(\Omega)}
<\eps$.
Now let $\delta>0$ be given by~\eqref{eq:AL_equi-cont} when applied to~$\psi$. 
Then, for all $s,t\in[0,T]$ such that $|s-t|<\delta$,
we have
\begin{align*}
|F_n[\phi](s)-F_n[\phi](t)|
&\le
|F_n[\psi](s)-F_n[\psi](t)|
+
2\|\psi-\phi\|_{L^\infty(\Omega)}
\,
\sup_{n\in\N}\|f^n\|_{L^\infty([0,T];L^1(\Omega))}
\\
&<
\eps
\,
\left(1+
2\sup_{n\in\N}\|f^n\|_{L^\infty([0,T];L^1(\Omega))}
\right)
\end{align*}
for all $n\in\N$, proving the validity of~\eqref{eq:AL_equi-cont} for $\phi\in C_c(\Omega)$.

\smallskip

\textit{Case 2}.
Let us prove~\eqref{eq:AL_equi-cont} for any $\phi\in L^\infty(\Omega)$.
Let $\eps>0$ be fixed and
let $\Omega_\eps\subset\Omega$ be the set given by~\eqref{eq:AL_eps-Omega_eps}.
Without loss of generality, we can assume that $\Omega_\eps$ is a non-empty open set.
Let $\delta'>0$ be given by~\eqref{eq:AL_eps-delta}. 
By the Lusin Theorem, we can find $\psi\in C_c(\Omega)$ with $\supp\psi\subset\Omega_\eps$ such that $\|\psi\|_{L^\infty(\Omega)}\le\|\phi\|_{L^\infty(\Omega)}$ and the set
\begin{equation*}
\tilde\Omega_\eps
=
\set*{x\in\Omega_\eps : \phi(x)=\psi(x)}
\subset
\Omega_\eps
\end{equation*} 
satisfies $|\Omega_\eps\setminus\tilde\Omega_\eps|<\delta'$.
Finally, let $\delta>0$ be given by~\eqref{eq:AL_equi-cont} when applied to~$\psi$.
Then, for all $s,t\in[0,T]$ such that $|s-t|<\delta$, we have
\begin{align*}
|F_n[\phi](s)-F_n[\phi](t)|
&\le
|F_n[\psi](s)-F_n[\psi](t)|
+
8\|\phi\|_{L^\infty(\Omega)}\,
\sup_{n\in\N}\|f^n\|_{L^\infty([0,T];L^1(\Omega\setminus\Omega_\eps))}
\\
&\quad+
8\|\phi\|_{L^\infty(\Omega)}\,
\sup_{n\in\N}\|f^n\|_{L^\infty([0,T];L^1(\Omega_\eps\setminus\tilde\Omega_\eps))}
\\
&<
\eps\,
\left(1+16\|\phi\|_{L^\infty(\Omega)}
\,
\sup_{n\in\N}\|f^n\|_{L^\infty([0,T];L^1(\Omega\setminus\Omega_\eps))}
\right)
\end{align*}
for all $n\in\N$, proving the validity of~\eqref{eq:AL_equi-cont} for $\phi\in L^\infty(\Omega)$.

\smallskip 

\textit{Step~2: definition of $f$ on a countable dense set $\mathcal T\subset[0,T]$}. By~\eqref{eq:AL_bounded}, \eqref{eq:AL_eps-delta} and~\eqref{eq:AL_eps-Omega_eps}, we can find a countable dense set ${\mathcal T}\subset [0,T]$ such that, for every given $t \in {\mathcal T}$, the sequence $(f^n(t,\cdot))_{n\in\N}$ is bounded in~$L^1(\Omega)$ and equi-integrable on $\Omega$. 
Therefore, by the Dunford--Pettis Theorem and a diagonal argument, we can find a subsequence $(f^{n_k})_{k \in \N}$ and a function $f(t,\cdot)\in L^1(\Omega)$, for each $t \in {\mathcal T}$, such that
\begin{equation}\label{convT}
\lim_{k\to+\infty}
\int_\Omega f^{n_k}(t,\cdot)\,\phi \di x
=
\int_\Omega f(t,\cdot)\,\phi \di x 
\end{equation}
for all $\phi\in L^\infty(\Omega)$ and $t\in\mathcal T$. 
We emphasize that the function  $f(t,\cdot)\in L^1(\Omega)$ depends on the chosen subsequence (which is fixed from now on) and is defined for $t \in {\mathcal T}$ only. 
Moreover, we have that 
\begin{equation}\label{eq:AL_propag}
(f(t,\cdot))_{t\in\mathcal T}\
\text{is bounded in~$L^1(\Omega)$ and equi-integrable on $\Omega$ uniformly in $t \in \mathcal T$}
\end{equation} 
thanks to the semicontinuity of the $L^1$-norm under weak$^\star$ convergence. 
Now, by Step~1, for each given $\phi \in L^\infty(\Omega)$, the sequence of functions $(F_{n_k}[\phi])_{k\in\N}$
is uniformly equi-continuous on $[0,T]$ and, thanks to~\eqref{convT}, it converges to the function
\begin{equation}\label{tbc}
{\mathcal T} \ni t \mapsto \int_\Omega f(t,\cdot)\,\phi \di x
\end{equation}
for each $t \in {\mathcal T}$.
Therefore, we must have that, for each given $\varphi \in L^\infty(\Omega)$, the function in~\eqref{tbc} is the restriction to ${\mathcal T}$ of a continuous function $F[\phi]\in C([0,T];\R)$.

\smallskip

\textit{Step~3: proof of~\eqref{eq:AL_prop_limit_f}}. 
We now extend the function $\mathcal T\ni t\mapsto f(t,\cdot)\in L^1(\Omega)$ given in Step~2 to a function $f \in C([0,T];L^1(\Omega)-w^\star)$. 
Let $t \in [0,T]\setminus{\mathcal T}$ be given. 
We claim that
\begin{equation}
\label{eq:AL_w-limitclaim}
\lim_{s\to t,\ s\in {\mathcal T}} f(s,\cdot)
\quad
\text{exists in}\ L^1(\Omega)-w^\star.
\end{equation}
In virtue of~\eqref{eq:AL_propag} and the Dunford--Pettis Theorem, we just need to prove that, for any two sequences $(t_m)_{m\in\N}\subset\mathcal T$ and $(\tilde t_m)_{m\in\N}\subset\mathcal T$ such that $
t_m,\tilde t_m\to t$ as $m\to+\infty$, 
\begin{equation*}
f(t_m,\cdot) \to g,
\
f(\tilde t_m,\cdot) \to \tilde g
\ 
\text{in}\ L^1(\Omega)-w^\star\
\text{as}\
m\to+\infty
\implies
g=\tilde g\
\text{in}\ L^1(\Omega).
\end{equation*}
Indeed, if $g\ne \tilde g$ in $L^1(\Omega)$ by contradiction, then we can find $\phi \in L^\infty(\Omega)$ such that
\begin{equation*}
\int_{\Omega} g \,\phi \di x
\not = 
\int_{\Omega} \tilde g \, \phi \di x.
\end{equation*}
However, since $f(t_m,\cdot) \to g$ and $f(\tilde t_m,\cdot) \to \tilde g$
in 
$L^1(\Omega)-w^\star$
as $m\to+\infty$,
this implies that
\begin{equation*}
\lim_{m\to+\infty}
\left|\, 
\int_{\Omega} f(t_m,\cdot) \,\phi \di x
-
\int_{\Omega} f(\tilde t_m,\cdot) \,\phi \di x 
\,\right|
>0,
\end{equation*}
which contradicts the continuity on~$\mathcal T$ of the function in~\eqref{tbc}.
This proves the claimed~\eqref{eq:AL_w-limitclaim} and thus the function $f\in C([0,T];L^1(\Omega)-w^\star)$ is well defined, meaning that, for each $\phi\in L^\infty(\Omega)$, we have
\begin{equation*}
t\mapsto\int_\Omega f(t,\cdot)\,\phi\di x
\in C([0,T];\R).
\end{equation*}
As a consequence, the function $f\colon[0,T]\to L^1(\Omega)$ is weakly measurable and thus, by the Pettis Theorem, is strongly measurable (for precise definitions and statements, see~\cite{L17}*{Chapter~8} and~\cite{RRS16}*{Section~1.9.1}), so that $f\in L^\infty([0,T];L^1(\Omega))$, with
\begin{equation*}
\|f\|_{L^\infty([0,T];\,L^1(\Omega))}
\le
\sup_{n\in\N}
\|f_n\|_{L^\infty([0,T];\,L^1(\Omega))},
\end{equation*}
and the function $f\colon[0,T]\times\Omega\to[-\infty,+\infty]$ is measurable.
This concludes the proof of~\eqref{eq:AL_prop_limit_f}.

\smallskip 

\textit{Step~4: proof of~\eqref{eq:AL_thesis}}.
We now conclude the proof by establishing  the convergence in~\eqref{eq:AL_thesis}. 
Let $\phi\in L^\infty(\Omega)$ and let $\eps>0$ be fixed. 
By Step~1, we can find $\delta>0$ such that
\begin{equation*}
\sup_{k\in\N}
\bigg| 
\int_\Omega f^{n_k}(t,\cdot)\,\phi \di x - \int_\Omega f^{n_k}(s,\cdot)\,\phi \di x
\,\bigg| 
<\frac\eps3
\end{equation*}
for all $s,t\in[0,T]$ such that $|s-t|<\delta$. 
Moreover, since $f\in C([0,T];L^1(\Omega)-w^\star)$ by Step~3, we can choose the above $\delta>0$ in such a way that, in addition,
\begin{equation*}
\bigg| 
\int_\Omega f(s,\cdot)\,\phi \di x 
- 
\int_\Omega f(t,\cdot)\,\phi \di x
\,\bigg| 
<\frac\eps3
\end{equation*} 
for all $s,t\in[0,T]$ such that $|s-t|<\delta$.
Now, since $[0,T]$ is a compact interval, we can find $N\in\N$ and $s_1,\dots,s_N\in\mathcal T$ such that
\begin{equation*}
[0,T]
=
\bigcup_{i=1}^N\set*{t\in[0,T] : |t-s_i|<\delta}.
\end{equation*}
Thanks to~\eqref{convT} in Step~2, for each $i=1,\dots,N$, we can choose $k_i\in\N$ such that 
\begin{equation*}
\bigg| 
\int_\Omega f^{n_k}(s_i,\cdot)\,\phi \di x 
- 
\int_\Omega f(s_i,\cdot)\,\phi \di x 
\,\bigg| 
<\frac\eps3
\end{equation*}
for all $k\ge k_i$.
Hence, let us set
$\bar k=\max\set*{k_i : i=1,\dots, N}$ and note that $\bar k$ depends on $\eps$ (and~$\phi$) only.
Now, given any $t\in[0,T]$, we can find $i\in\set*{1,\dots,N}$ such that $|t-s_i|<\delta$ and  
\begin{align*}
\bigg| 
&
\int_\Omega f^{n_k}(t,\cdot)\,\phi \di x 
- 
\int_\Omega f(t,\cdot)\,\phi \di x \,\bigg|
\le 
\bigg| 
\int_\Omega f^{n_k}(t,\cdot)\,\phi \di x - \int_\Omega f^{n_k}(s_i,\cdot)\,\phi \di x
\,\bigg| 
\\
&+ 
\bigg| 
\int_\Omega f^{n_k}(s_i,\cdot)\,\phi \di x 
- 
\int_\Omega f(s_i,\cdot)\,\phi \di x 
\,\bigg|
+ 
\bigg| 
\int_\Omega f(s_i,\cdot)\,\phi \di x 
- 
\int_\Omega f(t,\cdot)\,\phi \di x
\,\bigg|
<\eps 
\end{align*}
for all $k\ge\bar k$, proving the validity of~\eqref{eq:AL_thesis}. 
The proof is complete. 
\end{proof}


\end{document}